\documentclass[11pt,notitlepage,reqno,hidelinks]{amsart}

\usepackage[margin=1in]{geometry}
\usepackage{mathtools} % also loads amsmath
\usepackage{amssymb} % also loads amsfont
\usepackage{mathrsfs}
\usepackage{bbm}
\usepackage[shortlabels]{enumitem}
\usepackage{setspace}
\usepackage[dvipsnames]{xcolor}
\usepackage{multicol}
\usepackage{lastpage}
\usepackage{graphicx}
\usepackage{subfigure}
\usepackage{outlines}
\usepackage[labelfont=small,sc]{caption}
\usepackage{hyperref}
\usepackage[noabbrev,capitalise]{cleveref} % to follow amsthm, hyperref but before new theorem commands
\usepackage{float}

\usepackage{amsthm}
\newtheorem{theorem}{Theorem}[section]
\newtheorem*{theorem*}{Theorem}
\newtheorem{lemma}[theorem]{Lemma}
\newtheorem{proposition}[theorem]{Proposition}
\newtheorem*{proposition*}{Proposition}
\newtheorem{corollary}[theorem]{Corollary}
\newtheorem*{corollary*}{Corollary}
\newtheorem{notation}[theorem]{Notation}
\newtheorem{definition}[theorem]{Definition}

\newtheorem*{conjecture*}{Conjecture}
\newtheorem*{problem*}{Problem}

\theoremstyle{plain} % just in case the style had changed
\newtheorem{innercustomthm}{Theorem}
\newenvironment{customthm}[1]
  {\renewcommand\theinnercustomthm{#1}\innercustomthm}
  {\endinnercustomthm}
\crefname{innercustomthm}{Theorem}{Theorems}
\newtheorem{innercustomcor}{Corollary}
\newenvironment{customcor}[1]
  {\renewcommand\theinnercustomcor{#1}\innercustomcor}
  {\endinnercustomcor}
\crefname{innercustomcor}{Corollary}{Corollaries}
\theoremstyle{remark}
\newtheorem{remark}[theorem]{Remark}
\AtBeginEnvironment{remark}{%
  \pushQED{\qed}%
}
\AtEndEnvironment{remark}{\popQED\endremark}

\newtheorem{example}[theorem]{Example}
\AtBeginEnvironment{example}{%
  \pushQED{\qed}%
}
\AtEndEnvironment{example}{\popQED\endexample}
\newtheorem{tabexample}[theorem]{Ex.}

\newtheorem*{claim*}{Claim}

\numberwithin{equation}{section}
\numberwithin{theorem}{section}
\usepackage{chngcntr}
\allowdisplaybreaks
\makeatletter
\def\author@andify{%
	\nxandlist {\unskip ,\penalty-1 \space\ignorespaces}%
	{\unskip {} \@@and~}%
	{\unskip \penalty-2 \space \@@and~}%
}
\makeatother

\newlength{\truelen}
\newcommand{\padbox}[3][c]{%
    \settowidth{\truelen}{\ensuremath{#2}}%
    \ifdim\truelen < #3%
        \makebox[#3][#1]{\ensuremath{#2}}%
    \else%
        \ensuremath{#2}%
    \fi%
}

\newlength{\minlen}
\newcommand{\flexbox}[3][l]{%
    \settowidth{\minlen}{\ensuremath{#3}}%
    \padbox[#1]{#2}{\minlen}%
}

\usepackage{tikz}
\usetikzlibrary{matrix,chains,shapes,arrows,scopes,positioning}
\usetikzlibrary{decorations.markings,decorations.pathreplacing}
\usetikzlibrary{arrows,arrows.meta}
\usetikzlibrary{shapes.geometric}
\usetikzlibrary{patterns}
\tikzstyle{empty}=[circle,draw=black!80,thick]
\tikzstyle{emptyn}=[circle,draw=black!80,fill=white,scale=0.5] 
\tikzstyle{nero}=[circle,draw=black!80,fill=black!80,thick] 
\usetikzlibrary{positioning, fit, calc}

\newcommand{\Irr}{\operatorname{Irr}}
\newcommand{\MN}{Murnaghan--Nakayama }
\newcommand{\sgn}{\operatorname{sgn}}
\newcommand{\Syl}{\operatorname{Syl}}
\newcommand{\Van}{\operatorname{Van}}
\newcommand{\Vanpow}{\operatorname{Van}_{\mathrm{pow}}}
\newcommand{\down}{\big\downarrow}
\newcommand{\leglen}{\ell\ell}
\newcommand{\N}{\mathbb{N}}

\newcommand{\bw}{\mathbf{w}}

\newcommand{\cH}{\mathcal{H}}
\newcommand{\cP}{\mathcal{P}}

\newcommand{\sP}{\mathsf{P}}

\DeclarePairedDelimiter{\abs}{\lvert}{\rvert}
\newcommand{\la}{\lambda}
\newcommand{\eps}{\epsilon}
\renewcommand{\phi}{\varphi}
\renewcommand{\epsilon}{\varepsilon}
\renewcommand{\emptyset}{\varnothing}
\newcommand{\bipf}[1]{\underline{\smash{#1}}}
\newcommand{\nudge}{\addlinespace[8pt]}
\newcommand{\halfnudge}{\addlinespace[3pt]}

\newcommand{\midhat}[1]{\widehat{\smash{#1}\vphantom{1}}\vphantom{#1}}
\newcommand{\hatd}{\midhat{d}}
\newcommand{\hatnu}{\widehat{\nu}}
\newcommand{\hath}{\midhat{h}}
\newcommand{\hatw}{\widehat{w}}
\newcommand{\auxw}{\widehat{w}^{0}}

\usepackage{forest} % for drawing towers
\forestset{
    tower/.style={%
        for tree={
            l sep=6pt,
            anchor=center,
            s sep=17pt,
            execute at begin node={\(}, % so entering mathmode in every node is not necessary
            execute at end node={\)},
        }
    },
    compact tower/.style={%
        tower, for tree={
            s sep=6pt,
            l=6pt,
            execute at begin node={\Yboxdim{5pt} \scriptstyle},
        }
    },
    really compact tower/.style={%
        tower, for tree={
            s sep=0pt,
            l sep=4pt,
            l=2pt,
            execute at begin node={\Yboxdim{4pt} \scriptscriptstyle},
        }
    },
}
\usepackage{genyoungtabtikz}
\definecolor{paleazure}{HTML}{a6d2ff}
\definecolor{palegold}{HTML}{fff7bf}
\definecolor{beaverred}{HTML}{cc0000}
\definecolor{palered}{HTML}{f68c67}

\newcommand{\YcB}{\Yfillcolour{paleazure}}

\newcommand\wht{\Yfillcolour{white}}
\Yboxdim{9pt}
\usepackage{booktabs}
\usepackage{collcell}
\usepackage{makecell}
\usepackage{multirow}
\usepackage{array}   % for \newcolumntype macro
\newcolumntype{C}{>{$}c<{$}}

\crefformat{section}{\S#2#1#3}
\crefrangeformat{section}{\S#3#1#4--#5#2#6}
\crefmultiformat{section}{\S#2#1#3}{ and~\S#2#1#3}{, \S#2#1#3}{, and~\S#2#1#3}
\crefrangemultiformat{section}{\S#3#1#4--#5#2#6}{ and \S#3#1#4--#5#2#6}{, \S#3#1#4--#5#2#6}{ and \S#3#1#4--#5#2#6}

\crefformat{tabexample}{Example~#2#1#3}
\crefrangeformat{tabexample}{Examples~#3#1#4--#5#2#6}
\crefmultiformat{tabexample}{Examples~#2#1#3}{ and~#2#1#3}{, #2#1#3}{ and~#2#1#3}
\crefrangemultiformat{tabexample}{Examples~#3#1#4--#5#2#6}{ and #3#1#4--#5#2#6}{, #3#1#4--#5#2#6}{ and #3#1#4--#5#2#6}

\crefname{enumi}{}{} % remove "item" in \cref{[item]}

\begin{document}

% \vspace*{0.8cm} %0.8cm accounts for shift introduced by using \vspace
% \vspace*{-1.05cm}
\vspace*{-0.2cm}

\title[Degrees and prime power order zeros of characters of $S_n$ and $A_n$]{Degrees and prime power order zeros \\ of characters of symmetric and alternating groups}
% \date{}

\author{Eugenio Giannelli}
\address[E.~Giannelli]{Dipartimento di Matematica e Informatica	U.~Dini, Viale Morgagni 67/a, Università degli Studi di Firenze, Firenze, Italy}
\email{eugenio.giannelli@unifi.it}

\author{Stacey Law}
\address[S. Law]{School of Mathematics, Watson Building, University of Birmingham, Edgbaston, Birmingham B15 2TT, UK}
\email{s.law@bham.ac.uk}

\author{Eoghan McDowell}
\address[E.~McDowell]{School of Mathematics, University of Bristol, UK}
\email{eoghan.mcdowell@bristol.ac.uk}

\begin{abstract}
	We show that the $p$-part of the degree of an irreducible character of a symmetric group is completely determined by the set of vanishing elements of $p$-power order. 
    As a corollary we deduce that the set of zeros of prime power order controls the degree of such a character.
    The same problem is analysed for alternating groups, where we show that when \(p=2\) this data can only be determined up to two possibilities.
    We prove analogous statements for the defect of the $p$-block containing the character and for the $p$-height of the character.
    \looseness=-1
\end{abstract}

\dedicatory{Dedicated to Gabriel Navarro on his 60th birthday.}

% hijack \thanks field to use for journal publication information
\makeatletter
\def\@setthanks{%
\vspace{-\baselineskip}\vspace{4pt}
\def\thanks##1{\@par##1\@addpunct.}
\thankses%
}
\def\journalinfo#1{\thanks{#1}}
\makeatother

\journalinfo{
This is the accepted manuscript for an article published in the Bulletin of the London Mathematical Society, available online at \href{https://doi.org/10.1112/blms.70160}{https://doi.org/10.1112/blms.70160}.
}

\maketitle

\thispagestyle{empty}

\vspace*{-0.55cm}

%========================================================================
\section{Introduction}

Zeros of characters of finite groups have long interested mathematicians; a detailed account of the subject can be found in \cite{DPS}.
A recent striking result due to Miller \cite{Miller} is that, given a random irreducible character and a random element of a symmetric group \(S_n\), the probability that the character value equals zero is asymptotically equal to \(1\).
This has generated significant interest in the subject, see for example \cite{GLM, Moreto, Miller2025, Peluse2025}. 

Our focus in this article will be to examine how the zeros of \textbf{prime power order} determine properties of irreducible characters for symmetric and alternating groups.
Here, a \emph{zero} of a character \(\xi\) of a group \(G\) refers to an element \(g \in G\) such that \(\xi(g) = 0\).
We denote by $\Van(\xi)$ the set of all zeros of $\xi$ in $G$ (the \textit{vanishing set} of $\xi$).

The existence of a prime power order zero for every non-linear irreducible character of a finite group was established in \cite{MNO}, and some implications for the structure of Sylow subgroups were studied in \cite{DPSS,Madanha2022}.
Meanwhile, for any prime \(p\) and finite group \(G\) with Sylow \(p\)-subgroup \(P\), we know that the vanishing set $\Van(\chi\down_P)$ of an irreducible character \(\chi\) restricted to \(P\) is closely related to $\chi(1)_p$, the $p$-part of the degree of $\chi$ (i.e.~the largest power of \(p\) dividing \(\chi(1)\)).
The extreme cases are well understood: on the one hand, a theorem of Brauer asserts that $\chi(1)_p=|P|$ if and only if $\Van(\chi\down_P)=P\setminus\{1\}$ (see \cite[Theorem 8.17]{IBook});
on the other hand, it is well known that if $\chi(1)_p=1$ then $\Van(\chi\down_P)=\emptyset$. 
The converse to the last statement does not hold in general, but it was proved to be true in the case of symmetric and alternating groups \cite[Theorem 3.16]{GLLV}.
Another characterisation of when \(\chi(1)_p = 1\) in terms of zeros was given by Morotti \cite{Morotti}.
Our aim in this paper is to bridge these results by identifying the precise power of \(p\) dividing the character degree for symmetric groups and (as much as possible) for alternating groups.

The full vanishing set of an irreducible character of the symmetric group was shown to determine the character up to multiplication by linear characters by Belonogov \cite{Belonogov} (that is, \(\Van(\chi)\) determines \(\chi\) up to multiplication by the sign character).
In particular, one has that the full vanishing set $\Van(\chi)$ determines the degree $\chi(1)$.
G.~Navarro conjectured that, for the symmetric group, $\chi(1)$ should be fully controlled just by the set of prime power order zeros of $\chi$. 
We confirm this statement in \Cref{thm:main-sym-compare-allp} below.

The vanishing behaviour for characters of the alternating group is known to be more complicated: classifying which irreducible characters of the alternating group have the same zeros is an open problem, with a conjecture due to C.~Bessenrodt reported by Bowman \cite{MFO}.
Nevertheless, we are able to show that the degree is determined by the prime power order zeros up to two possibilities, providing examples to illustrate each possible behaviour. Furthermore, we demonstrate that this cannot be improved.

\medskip

\subsection{Results for symmetric groups \texorpdfstring{\(S_n\)}{Sn}}

\begin{customthm}{S1}\label{thm:main-sym-explicit}
	Let $n\in\N$ and let $p$ be a prime.
	Let \(P \in \Syl_p(S_n)\) and let \(\chi \in \Irr(S_n)\).
	The following numbers can be explicitly determined from the vanishing set $\Van(\chi\down_P)$:
	\begin{enumerate}[(i)]
	    \item the \(p\)-part of the degree of \(\chi\); % \(\chi(1)_p\)
	    \item the defect of the \(p\)-block containing \(\chi\);
	    \item the $p$-height of \(\chi\) in its $p$-block. % \(h_p(\chi)\).
	\end{enumerate}
\end{customthm}

We prove \Cref{thm:main-sym-explicit} in \Cref{sec:sym} by expressing the desired numbers in terms of the \emph{prime power weights} of the labelling partition of the character, which in turn are explicitly determined from the vanishing set in \Cref{thm:determine_weights}.
Our explicit determination has the following consequence.

\begin{customthm}{S2}\label{thm:main-sym-compare-pparts}
	Let $n\in\N$ and let $p$ be a prime.
	Let \(P \in \Syl_p(S_n)\) and let \(\chi,\psi \in \Irr(S_n)\).
	\[ \text{If}\ \Van(\chi\down_P)\subseteq\Van(\psi\down_P),\ \text{then}\ \chi(1)_p\leq\psi(1)_p. \]
In particular, if \(\Van(\chi\down_P)=\Van(\psi\down_P)\), then \(\chi(1)_p=\psi(1)_p\).
\end{customthm}

We remark that the converse of \Cref{thm:main-sym-compare-pparts} does not hold; counterexamples for every prime are provided in \cref{eg:converse_fails,eg:converse_fails_p=2}.
Nor does \Cref{thm:main-sym-compare-pparts} hold if the containments and inequalities are made strict; a counterexample for \(p=2\) is given by the irreducible characters of \(S_{10}\) labelled by the partitions \((6,4)\) and \((6,2,1,1)\).

By considering \Cref{thm:main-sym-compare-pparts} for all primes \(p\), we deduce the following corollary, that the set of zeros of prime power order of an irreducible character $\chi$ of $S_n$ determines the degree of $\chi$, confirming Navarro's prediction.
For a character \(\chi\), let \(\Vanpow(\chi) := \{g \in \Van(\chi) \mid g \text{ has prime power order}\}\).

\begin{customcor}{S3}\label{thm:main-sym-compare-allp}
	Let $n\in\N$ and let \(\chi,\psi \in \Irr(S_n)\).
	\[ \text{If}\ \Vanpow(\chi) \subseteq \Vanpow(\psi),\ \text{then}\ \chi(1)\leq\psi(1). \]
In particular, if \(\Vanpow(\chi)=\Vanpow(\psi)\), then \(\chi(1)=\psi(1)\).
\end{customcor}

\Cref{thm:main-sym-explicit} can be strengthened by replacing the Sylow subgroup \(P\) by any subgroup containing a defect group of the block in which \(\chi\) lies.
This is explained in \Cref{subsec:reduction_to_defect_group}.
As a consequence of the strengthened theorem, we deduce the following blockwise version of \Cref{thm:main-sym-compare-pparts}.
Given a $p$-block $B$ of $S_n$, we denote by $\mathrm{Irr}(B)$ the set of irreducible characters lying in $B$, and by $h_p(\chi)$ the $p$-height of an irreducible character $\chi$ in its $p$-block
(see \Cref{subsec:blocks} for definitions).
% (which we recall is related to $\chi(1)_p$ via \eqref{eq:height}).

\begin{samepage}%
\begin{customcor}{S4}\label{thm:main-sym-compare-heights}
Let $n\in\N$ and $p$ be any prime. Let $B$ be a $p$-block of $S_n$ with defect group $D$ and let $\chi,\psi\in\Irr(B)$.
    \[ \text{If}\ \Van(\chi\down_D)\subseteq\Van(\psi\down_D),\ \text{then}\ h_p(\chi) \leq h_p(\psi). \]
In particular, if \(\Van(\chi\down_D)=\Van(\psi\down_D)\), then \(h_p(\chi)=h_p(\psi)\).
\end{customcor}
\end{samepage}

\medskip

\subsection{Results for alternating groups \texorpdfstring{\(A_n\)}{An}}

We consider the same question for the alternating group.
When the prime \(p\) is odd, the answer is essentially the same as, and follows very quickly from, that for the symmetric group.
However, when \(p=2\), the data of the character can in general only be determined up to two possibilities, and the proof is significantly more technical. 

\begin{customthm}{A1}\label{thm:main-alt-explicit}
	Let $n\in\N$ and let $p$ be a prime.
	Let \(Q \in \Syl_p({A_n})\) and let \(\chi \in \Irr(A_n)\).
	The following numbers can be explicitly determined from the vanishing set $\Van(\chi\down_{Q})$, except when \(p=2\) when they can be determined up to two possibilities:
	\begin{enumerate}[(i)]
	    \item the \(p\)-part of the degree of \(\chi\); % \(\chi(1)_p\)
	    \item the defect of the \(p\)-block containing \(\chi\);
	    \item the $p$-height of \(\chi\) in its $p$-block. % \(h_p(\chi)\).
	\end{enumerate}
\end{customthm}

The explicit method of determination is detailed in \Cref{cor:alt-determine-data}.
We remark that, when \(p=2\), there are many cases when the numbers above can in fact be uniquely determined (for example, the defect can be uniquely determined whenever \(n \not\equiv 3 \pmod{4}\)).
Nevertheless, in general, the statement cannot be improved: there exists pairs of irreducible characters with equal vanishing sets on \(Q\) but with differing \(2\)-heights, defects and/or \(2\)-parts of degrees. 
Several examples of this are presented in \cref{sec:examples}.

The explicit determination in \Cref{cor:alt-determine-data} has the following consequence.

\begin{customthm}{A2}\label{thm:main-alt-compare-pparts}
Let $n\in\N$ and let $p$ be a prime.
Let $Q\in\mathrm{Syl}_p(A_n)$ and let $\eta, \theta\in\Irr(A_n)$.
\[ \text{If}\ \Van(\eta\down_Q)\subseteq\Van(\theta\down_Q),\ \text{then}\ \begin{cases}
	    \eta(1)_p \leq \theta(1)_p & \text{ if \(p\) is odd;} \\
        \eta(1)_2 \leq 2\cdot\theta(1)_2 & \text{ if \(p=2\).}
	\end{cases} 
\]
In particular,
suppose \(\Van(\eta\down_Q)=\Van(\theta\down_Q)\); if \(p\) is odd, then \(\eta(1)_p=\theta(1)_p\),
while if \(p=2\) and without loss of generality \(\eta(1)_2 \geq \theta(1)_2\), then $\eta(1)_2 \in \{\,\theta(1)_2,\, 2\cdot\theta(1)_2\,\}$.
\end{customthm}

As with the symmetric group, we obtain corollaries by by considering all primes simultaneously and by weakening the hypothesis to only require knowledge of the vanishing on a defect group.

\begin{customcor}{A3}\label{thm:main-alt-compare-allp}
	Let $n\in\N$ and let \(\eta,\theta \in \Irr(A_n)\).
	\[ \text{If}\ \Vanpow(\eta)\subseteq\Vanpow(\theta),\ \text{then}\ \eta(1)\leq2\cdot\theta(1). \]
Moreover, if \(\Vanpow(\eta)=\Vanpow(\theta)\) and without loss of generality \(\eta(1) \geq \theta(1)\), then $\eta(1) \in \{\,\theta(1),\, 2\cdot\theta(1)\,\}$.
\end{customcor}

\begin{customcor}{A4}\label{thm:main-alt-compare-heights}
Let $n\in\N$ and $p$ be a prime. Let $B$ be a $p$-block of $A_n$, let \(D\) be a defect group of \(B\)
    and let $\eta, \theta\in\Irr(B)$.
	\[ \text{If}\ \Van(\eta\down_D)\subseteq\Van(\theta\down_D),\ \text{then}\ \begin{cases}
	    h_p(\eta) \leq h_p(\theta) & \text{ if \(p\) is odd;} \\
        h_2(\eta) \leq h_2(\theta)+1 & \text{ if \(p=2\).}
	\end{cases}  \]
In particular,
suppose \( \Van(\eta\down_D)=\Van(\theta\down_D)\); if \(p\) is odd, then \(h_p(\eta) = h_p(\theta)\),
while if \(p=2\) and without loss of generality \(h_p(\eta) \geq h_p(\theta)\), then $h_2(\eta) \in \{\,h_2(\theta),\, h_2(\theta)+1\,\}$.
\end{customcor}

\medskip

\subsection{Reduction to defect groups}
\label{subsec:reduction_to_defect_group}

The following proposition tells us that the interesting vanishing behaviour of an irreducible character occurs on a defect group for its block.
(The proofs for \(S_n\) and \(A_n\) can be found easily using the description of their defect groups
in \Cref{subsec:blocks}.)

\begin{proposition}[{\cite[Corollary~5.9]{NavarroBook}}]
	Let \(G\) be a finite group, let \(B\) be a block of \(G\) and let \(\chi \in \Irr(B)\).
	Let \(g \in G\).
	If the \(p\)-part of \(g\) does not lie in a defect group of \(B\), then \(\chi(g) = 0\).
\end{proposition}

This justifies our claim that we can strengthen our main theorems to only requiring knowledge of the vanishing set on a subgroup containing the defect group, and hence deduce \Cref{thm:main-sym-compare-heights,thm:main-alt-compare-heights}.
We nevertheless state our other main results using the full Sylow subgroup to avoid assuming that we {\itshape a priori} know the blocks or their defect groups.

We can in fact weaken the hypotheses of \Cref{thm:main-sym-compare-heights} further to permit $\chi$ and $\psi$ to lie in different blocks (and similarly for \Cref{thm:main-alt-compare-heights}).
More precisely, suppose \(\chi\) lies in a block \(B\) of \(S_n\) with defect group \(D\) and suppose \(\psi\) lies in a block \(B'\) of \(S_n\) with defect group \(D'\), and suppose that \(|D| \geq |D'|\), so that, in particular, $D'$ is $S_n$-conjugate to a subgroup of $D$ (see \cref{subsec:blocks}).
If \(\Van(\chi\down_D) \subseteq \Van(\psi\down_D)\) then we can deduce \(\chi(1)_p \leq \psi(1)_p\), while
if \(\Van(\chi\down_D) = \Van(\psi\down_D)\) then we can deduce that \(|D| = |D'|\), \(\chi(1)_p = \psi(1)_p\) and $h_p(\chi)=h_p(\psi)$.

\medskip

\subsection{What the prime power order zeros do \emph{not} know}

At the beginning of the introduction, we noted that given $\chi, \psi \in \Irr(S_n)$, Belonogov's work \cite{Belonogov} demonstrates that if $\Van(\chi) = \Van(\psi)$, then necessarily $\chi = \psi \cdot \zeta$ for some linear character $\zeta$ of $S_n$.
In contrast,
\Cref{thm:main-sym-compare-pparts} establishes that if $\Vanpow(\chi) = \Vanpow(\psi)$, then $\chi(1) = \psi(1)$.
This raises the natural question of whether $\Vanpow(\chi)$ suffices to uniquely determine the irreducible character $\chi$, up to multiplication by a linear character. The answer to this question is negative.
As a counterexample, consider the irreducible characters $\chi^\lambda, \chi^\mu \in \mathrm{Irr}(S_{14})$ corresponding to the partitions $\lambda = (6, 3, 3, 2)$ and $\mu = (5, 5, 2, 1, 1)$. Direct computations reveal that $\Vanpow(\chi^\lambda) = \Vanpow(\chi^\mu)$, yet $\chi^\lambda$ is not equal to $\chi^\mu$ multiplied by any linear character of $S_{14}$.
This underscores the distinction between \Cref{thm:main-sym-compare-pparts} and Belonogov's result.

For the alternating group, we already noted that the vanishing set on a Sylow \(2\)-subgroup is insufficient to determine the \(2\)-part of the degree.
This ambiguity cannot be resolved even if we consider all prime power order elements.
Consider $\lambda = (10, 4, 3)$ and $\mu = (7, 2, 2, 2, 2, 2)$,
and let $\eta = \chi^\lambda \down_{A_{17}}$ and $\theta = \chi^\mu \down_{A_{17}}$ denote the corresponding irreducible characters of $A_{17}$.
Direct computations show that $\Vanpow(\eta) = \Vanpow(\theta)$ and that $\theta(1) = 2 \cdot \eta(1) \neq \eta(1)$.
This demonstrates that, for alternating groups, the set of prime power order zeros of an irreducible character does not uniquely determine its degree.
 
\bigskip

\section{Background on combinatorics and characters}

Throughout, we let $\N$ denote the set of natural numbers and $\N_0$ the non-negative integers.
For $n\in\N$ and $p$ a prime, we denote by $n_p$ the maximal power of $p$ dividing $n$, and by $\nu_p(n)$ the $p$-adic valuation of $n$ (so that $p^{\nu_p(n)}=n_p$).

We recall some basic properties of the irreducible characters of $S_n$ and \(A_n\), and refer the reader to \cite{JK} and \cite{OlssonBook} for more detailed discussions of the topic. 

\medskip

\subsection{Partitions and characters}

A \emph{partition} of \(n \in \N_0\) is a sequence of weakly decreasing positive integers whose sum is \(n\).
When writing partitions, we frequently collect parts of equal size, letting \(e^w\) denote \(w\) parts of size \(e\).
(It should always be clear from context, e.g.~by specifying the size of the partition, whether $e^w$ refers to a single part of size $e^w$ or $w$ parts of size $e$.)
For each $n\in\N_0$, we let $\cP(n)$ denote the set of partitions of $n$.
We write $|\lambda|=n$ to mean $\lambda\in\cP(n)$.
The \emph{Young diagram} $Y(\lambda)$ of a partition $\lambda=(\lambda_1,\ldots, \lambda_z)$ is defined as
\[ Y(\lambda)=\{(i,j)\in\mathbb{N}\times\mathbb{N} \mid 1\leq i\leq z,\ 1\leq j\leq\la_i\}. \]
Each element $(i,j)\in Y(\lambda)$ is called a \textit{node} of $Y(\lambda)$.
We denote by $\lambda'$ the conjugate partition of $\lambda$ which is the partition satisfying $Y(\lambda')=\{(j,i)\mid(i,j)\in Y(\lambda)\}$.

The irreducible characters of $S_n$ are naturally parametrised by the elements of $\cP(n)$.
Given a partition $\lambda$ of $n$ we denote by $\chi^{\lambda}$ its corresponding irreducible character.

The irreducible characters of $A_n$ can be easily described in relation to those of $S_n$.
For any $\lambda\in\cP(n)$ such that $\lambda\neq\lambda'$, we have that $\chi^\lambda\down_{A_n}=\chi^{\lambda'}\down_{A_n}\in\Irr(A_n)$. On the other hand, if $\lambda=\lambda'$ then $\chi^\lambda\down_{A_n}=\phi^{\lambda+}+\phi^{\lambda-}$ with $\phi^{\lambda+}\ne\phi^{\lambda-}\in\Irr(A_n)$. All of the irreducible characters of $A_n$ are of one of these two forms. In other words,
\[ \Irr(A_n) = \{ \chi^\lambda\down_{A_n} \mid \lambda\ne\lambda'\in\cP(n) \} \sqcup \{ \phi^{\lambda+},\phi^{\lambda-} \mid \lambda=\lambda'\in\cP(n) \}.  \]
Given $\phi\in\Irr(A_n)$ and $\chi\in\Irr(S_n)$, we say that $\chi$ \textit{covers} $\phi$ if \(\phi\) is a constituent of \(\chi\down_{A_n}\). 

\medskip

\subsection{Hooks}

The \textit{hook} of $\lambda$ associated to the node $(i,j)$ is the set of nodes
\[
    H_{i,j}(\lambda) = \{\, (i,j) \,\} \sqcup \{\, (i,y) \mid j+1 \leq y \leq \lambda_i \,\} \sqcup \{\, (x,j) \mid i+1 \leq x \leq \lambda'_j\,\},
\]
and we let $h_{i,j}(\lambda)=|H_{i,j}(\lambda)|$ be the \emph{length} of the hook $H_{i,j}(\lambda)$. 
Following standard notation, we use the symbol $\cH(\lambda)$ to denote the multiset of hook lengths of $\lambda$.
We remark that in this article, the term `$e$-hook' will always mean a hook of length equal to $e$.

The hook length formula (stated below; see also \cite[Theorem~2.3.21]{JK}) expresses the degree of an irreducible character in terms of hook lengths.

\begin{theorem}[Hook length formula]\label{thm:hlf}
	Let \(n \in \N\) and let \(\la\) be a partition of \(n\).
	Then
	\[ \chi^\la(1) = \frac{n!}{\prod_{h \in \cH(\la)}h}. \]
\end{theorem}

\medskip

\subsection{Cores, quotients and weights}

A hook can be \emph{removed} from a partition by deleting the nodes of the hook from the Young diagram and sliding the remaining nodes north-west so that they form a new Young diagram.
This is illustrated below with the removal of the \(8\)-hook \(H_{2,3}(\la)\) (shaded in the first diagram) from the partition \(\la = (9,8,6,5,1)\) to form the partition \(\la \setminus H_{2,3}(\la) = (9,5,4,2,1)\).

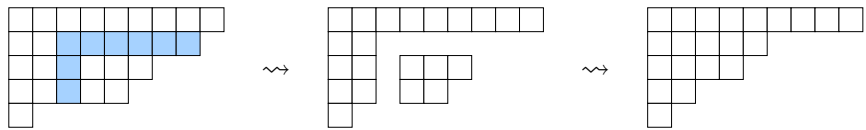
\[
\gyoung(;;;;;;;;;,;;!\YcB;;;;;;,!\wht;;!\YcB;!\wht;;;,;;!\YcB;!\wht;;,;)
\quad\rightsquigarrow\quad
\gyoung(;;;;;;;;;,;;::::::,;;:;;;,;;:;;,;)
\quad\rightsquigarrow\quad
\gyoung(;;;;;;;;;,;;;;;,;;;;,;;,;)
\]
\vspace{0.1cm}

Let \(n \in \N\), \(e \in \N_{\geq 2}\) and let \(\la \in \cP(n)\).
The \emph{\(e\)-core} of a partition \(\la\), which we denote by \(C_e(\la)\), is the partition obtained by iteratively removing all \(e\)-hooks from \(\la\).
(This turns out to be independent of the order in which $e$-hooks are removed; see \cite[Theorem 2.7.16]{JK}, for instance.)
The \(e\)-weight of \(\la\), which we denote by \(\bw_e(\la)\), is the number of hooks removed from \(\la\) to form \(C_e(\la)\); equivalently, \(\bw_e(\la)\) is the number of hooks in \(\la\) of length divisible by \(e\).
Note \(n= |C_e(\la)| + e\bw_e(\la)\).

The \emph{\(e\)-quotient} of \(\la\), which we denote by $Q_e(\lambda)=(\lambda^{(0)},\lambda^{(1)},\dotsc,\lambda^{(e-1)})$, is an \(e\)-tuple of partitions that encodes the structure of the \(e\)-hooks of \(\la\): see \cite[Chapter 2]{JK} or \cite[Section 3]{OlssonBook} for details on how to define the \(e\)-quotient via James's abacus (we record that, as in \cite[Section 3]{OlssonBook}, when calculating $e$-quotients we adopt the convention of using abaci which have a multiple of $e$ many beads).
For our purposes, it is sufficient to note the following key property.

\begin{proposition}[{\cite[Theorem 3.3]{OlssonBook}}]\label{prop:quotient-hooks}
	Let $e\in\N$ and $\lambda$ be a partition. There is a canonical bijection $f$ from the multiset of hooks of $\lambda$ of length divisible by $e$, to the multiset of hooks in $Q_e(\lambda)$ (meaning the union of the multisets of hooks in $\lambda^{(0)},\dotsc,\lambda^{(e-1)}$). Moreover, $|H|=e\cdot |f(H)|$ for each such hook $H$ of $\lambda$, and \(Q_e(\la \setminus H) = Q_e(\la) \setminus f(H)\).
\end{proposition}

In particular, the \(e\)-quotient has size \(|Q_e(\lambda)| = |\lambda^{(0)}|+|\lambda^{(1)}|+\dotsc+|\lambda^{(e-1)}|\) equal to the number of hooks of $\lambda$ of length divisible by \(e\); that is, \(\bw_e(\la) = |Q_e(\lambda)|\).

For taking cores and quotients iteratively, we use the following notation.
For \(\bipf{\la}\) an \(e\)-tuple of partitions, we write \(C_r(\bipf{\la})\) for the \(e\)-tuple obtained by replacing every partition with its \(r\)-core, and we write \(Q_r(\bipf{\la})\) for the \(er\)-tuple obtained by replacing every partition with its \(r\)-quotient and concatenating.
The following facts can be deduced from \Cref{prop:quotient-hooks}, or can be seen readily from their construction on James's abacus (indeed, part (ii) can be strengthened to the observation that \(Q_r(Q_e(\la))\) is a permutation of \(Q_{er}(\la)\)).

\begin{proposition}
\label{prop:iterated_cores_and_quotients}
	Let \(n \in \N\), \(\la \in \cP(n)\) and \(e, r \geq 2\).
	Then
	\begin{enumerate}[(i)]
	    \item\label{item:corequot_vs_quotcore} \(C_r(Q_e(\la)) = Q_e(C_{er}(\la))\);
	    \item \(\abs{Q_r(Q_e(\la))} = \abs{Q_{er}(\la)}\).
	\end{enumerate}
\end{proposition}

\medskip

\subsection{Blocks and defects}\label{subsec:blocks}

Let us now fix a prime number $p$ and a finite group $G$. Let $B$ be a $p$-block of $G$. The \emph{defect} of $B$ is the non-negative integer $d$ such that $|D| = p^d$, where $D$ is a defect group of $B$. If $B$ has defect $d$ and $\chi\in\Irr(B)$, then the \emph{$p$-height} of $\chi$ is the non-negative integer \(h_p(\chi)\) such that 
\begin{equation*}
	\nu_p(\chi(1)) = a-d+h_p(\chi)
\end{equation*}
where \(p^a\) is the order of a Sylow \(p\)-subgroup of \(G\).

It is well known that $\chi^\lambda$ and $\chi^\mu$ lie in the same $p$-block of $S_n$ if and only if $C_p(\lambda)=C_p(\mu)$ (see \cite[6.1.21]{JK}, for instance). In other words, the $p$-blocks of $S_n$ are naturally parametrised by $p$-core partitions.
The defect groups of the \(p\)-block containing \(\chi^\la\) are the Sylow $p$-subgroups of $S_{p\bw_p(\la)}$, viewed as subgroups of $S_n$ in the natural way (by permuting $p\bw_p(\lambda)$ of the $n$ points on which $S_n$ acts) (see \cite[6.2.45]{JK}).
Thus the defect of the block containing \(\chi^\la\) is \(\nu_p((p\bw_p(\la))!)\).

The block decomposition for the alternating group is discussed in \cite{OlssonBlocks}.
If $B$ is a $p$-block of $S_n$ of positive defect, then $B$ covers a unique block $b$ of $A_n$.
Moreover, the defect groups of $b$ are all the $A_n$-conjugates of $D\cap A_n$, where $D$ is a defect group of $B$.
In particular, the defect groups of $b$ and $B$ are isomorphic unless $p=2$ and $B$ has positive defect, in which case $|D: D\cap A_n|=2$.

\medskip

\subsection{The Murnaghan--Nakayama rule}

The \emph{\MN rule} (stated below; see also \cite[2.4.7]{JK}) describes how to recursively compute any character value in terms of hook lengths. 
The \emph{leg length} $\leglen(H_{i,j}(\lambda))$ of a hook \(H_{i,j}(\lambda)\) is equal to one less than the number of rows of $Y(\lambda)$ it occupies; in other words, $\leglen(H_{i,j}(\lambda))=\lambda'_j-i$.

\begin{theorem}[\MN rule]\label{thm:MN}
	Let $e,n\in\N$ with $e<n$, and let \(\la \in \cP(n)\).
    Let $\rho \in S_n$ be an $e$-cycle and let $\pi \in S_n$ be a permutation of the remaining $n-e$ numbers. Then
\[
    \chi^\lambda(\pi\rho) = \sum_\mu (-1)^{\leglen(\lambda\setminus\mu)}\chi^\mu(\pi)
\]
	where the sum runs over all partitions $\mu$ obtained from $\lambda$ by removing an $e$-hook.
    In particular, if $\lambda$ has no $e$-hooks, then $\chi^\lambda(\pi\rho)=0$.
\end{theorem}

The following result from \cite[Proposition 3.13]{OlssonBook} is fundamental in understanding the parities of leg lengths in hook removal processes on partitions.

\begin{lemma}\label{lem:path-sign}
	Let $e\in\N$. Suppose $\lambda$ and $\mu$ are two partitions such that $\mu$ is obtained from $\lambda$ by removing a sequence of $r$ many $e$-hooks for some $r\in\N_0$, having leg lengths $b_1,b_2,\dotsc,b_r$ respectively. Let $b=\sum_i b_i$. Then the residue of $b$ (mod 2) does not depend on the choice of $e$-hooks being removed in going from $\lambda$ to $\mu$.
\end{lemma}

We introduce some notation to discuss different ways of removing successive hooks. % from a partition.
	
\begin{notation}\label{not:path}
	Let $e\in\N$ and $\lambda$ and $\mu$ be two partitions such that $\mu$ may be obtained from $\lambda$ by successively removing some number of $e$-hooks. 
	We use the notation 
	\[ P:\lambda\xrightarrow{e}\mu \]
	to mean that $P$ is a specific sequence of hook removals, which we call a \emph{path}, starting at $\lambda$ and arriving at $\mu$ where each successive hook removed has length $e$.
	Equivalently, $P$ is a sequence of partitions $\big(\alpha(0),\alpha(1),\dotsc,\alpha(v)\big)$ for some $v\in\N_0$ satisfying the following: $\alpha(0)=\lambda$, $\alpha(v)=\mu$, and $\alpha(i+1)$ is obtained by removing an $e$-hook $H_i$ from $\alpha(i)$ for each $i\in\{0,1,\dotsc,v-1\}$ (that is, $\alpha(i+1)=\alpha(i)\setminus H_i$).
	
	For such a path $P$, we define the \emph{sign} of $P$, denoted $\sgn(P)$, to be
	\[ \sgn(P):= (-1)^{\sum_{i=0}^{v-1} \leglen(H_i)}. \]
\end{notation}

We observe by \cref{lem:path-sign} that the value of $\sgn(P)$ in fact depends only on $\lambda$, $\mu$ and $e$, and not on the path $P$ itself.
We thus define \(\sgn(\la\setminus\mu)\) to be the sign of any path \(P \colon \lambda \xrightarrow{e} \mu\) (where the parameter \(e\) will be understood from context).
We can therefore reformulate the \MN rule in terms of paths as follows.

\begin{theorem}[\MN rule reformulated]\label{thm:MN-reformulated}
Let $e,r,n\in\N$ with $er<n$, and let \(\la \in \cP(n)\).
Let $\rho \in S_n$ be a product of \(r\)-many disjoint $e$-cycles and let $\pi \in S_n$ be a permutation of the remaining $n-er$ numbers. Then
\[
    \chi^\lambda(\pi\rho) = \sum_\mu \sgn(\lambda\setminus\mu) \, |\{P \mid P\colon \la \xrightarrow{e} \mu \}| \, \chi^\mu(\pi)
\]
where the sum runs over all partitions $\mu$ obtained from $\lambda$ by removing a sequence of \(r\)-many $e$-hooks.
In particular, if $\lambda$ has no $e$-hooks, then $\chi^\lambda(\pi\rho)=0$.
\end{theorem}

\bigskip

%--------------------------------------------------------------------------
\section{Main results for the symmetric group}\label{sec:sym}

The goal of this section is to prove \Cref{thm:main-sym-explicit} by showing that the vanishing set on a Sylow subgroup determines the prime power weights of the labelling partition.
We begin by showing how to use the \MN rule to identify non-zero character values.

\begin{lemma}\label{lemma:unique_partition_implies_non-zero}
Let $e,r,n\in\N$ with $er<n$, and let \(\la \in \cP(n)\).
Let $\rho \in S_n$ be a product of \(r\)-many disjoint $e$-cycles and let $\pi \in S_n$ be a permutation of the remaining $n-er$ numbers.
    Suppose there is a unique partition $\mu$ that can be obtained from $\lambda$ by removing $r$ many $e$-hooks.
    Then $\chi^\lambda(\pi\rho)$ is a non-zero multiple of $\chi^\mu(\pi)$.
    In particular, \(\chi^\la(\rho) \neq 0\).
\end{lemma}

\begin{proof}
This is clear from our reformulation of the \MN rule (\Cref{thm:MN-reformulated}).
\end{proof}

This suffices to determine the prime power weights from the vanishing set. %of an irreducible character.

\begin{theorem}
\label{thm:determine_weights}
	Let $n\in\N$, $\lambda\in\cP(n)$ and $e\in\N_{\ge 2}$. The $e$-weight $\bw_e(\lambda)$ of $\lambda$ is the maximal $w\in\N_0$ such that $\chi^\lambda$ is non-zero on an element of cycle type $(e^w,1^{n-ew})$.
\end{theorem}

\begin{proof}
	The $e$-core of $\lambda$ is the unique partition that can be obtained from $\lambda$ by removing $\bw_e(\lambda)$ many $e$-hooks, so $\chi^\lambda$ is non-zero on elements of cycle type $(e^{\bw_e(\lambda)},1^{n-e\bw_e(\lambda)})$ by \cref{lemma:unique_partition_implies_non-zero}.
    Meanwhile, for any $r>\bw_e(\lambda)$, there is no partition that can be obtained from $\lambda$ by removing $r$ many $e$-hooks, so $\chi^\lambda$ vanishes on elements of cycle type $(e^r,1^{n-er})$ by \cref{thm:MN}.
\end{proof}

It remains now only to express our three desired pieces of information -- the \(p\)-part of the degree; the defect of the block; and the height -- in terms of the prime power weights.

\begin{lemma}
\label{lemma:ppart_from_weights}
Let \(n \in \N\), let \(p\) be a prime, and let \(\la \in \cP(n)\).
Suppose $n=\sum_{r\ge 0}b_rp^r$ is the $p$-adic expansion of $n$ (i.e.~so that $b_r\in\{0,1,\ldots,p-1\}$ for all $r\in\N_0$).
Then
\[
    \nu_p(\chi^\lambda(1)) = \frac{n - \sum_{r\ge 0}b_r}{p-1} - \sum_{i \geq 1}\bw_{p^i}(\la).
\]
\end{lemma}

\begin{proof}
By the hook length formula (\Cref{thm:hlf}) we have \(\nu_p(\chi^\la(1)) = \nu_p(n!) - \nu_p(\prod_{h \in \cH(\la)} h)\).
We have \(\nu_p(n!) = \frac{n - \sum_{r\ge 0}b_r}{p-1}\) by Legendre's formula.
Meanwhile, since \(\bw_{p^i}(\la)\) is the number of hooks of length divisible by \(p^i\), we have \(\nu_p(\prod_{h \in \cH(\la)} h) = \sum_{i \geq 1}\bw_{p^i}(\la)\).
\end{proof}

\begin{proof}[Proof of \Cref{thm:main-sym-explicit}]
Let \(\la \in \cP(n)\) be the partition labelling \(\chi\).
We write \(a = \nu_p(n!) = \frac{n - \sum_{r\ge 0}b_r}{p-1}\), where $n=\sum_{r\ge 0}b_rp^r$ is the $p$-adic expansion of $n$, so that \(|P| = p^a\).
\Cref{thm:determine_weights} shows that the prime power weights \(\bw_{p^i}(\la)\) can be explicitly determined from \(\Van(\chi\down_P)\).
The desired numbers can then be expressed as follows:
\begin{enumerate}[(i)]
    \item the \(p\)-part of the degree is \(\chi(1)_p = p^{a - \sum_{i \geq 1}\bw_{p^i}(\la)}\) by \Cref{lemma:ppart_from_weights};
    \item the defect of the block containing \(\chi\) is \(d = \nu_p( (p\bw_p(\la))! )\), since the defect groups are Sylow $p$-subgroups of \(S_{p\bw_p(\la)}\);
    \item the height of \(\chi\) is \(h_p(\chi) = d - \sum_{i \geq 1} \bw_{p^i}(\la)\), since it satisfies \(\chi(1)_p =p^{a-d + h_p(\chi)}\). \qedhere
\end{enumerate}
\end{proof}

\begin{proof}[Proof of \Cref{thm:main-sym-compare-pparts}]
Let \(\la,\mu \in \cP(n)\) be the partitions labelling \(\chi,\psi\) respectively.
Suppose \(\Van(\chi\down_P) \subseteq \Van(\psi\down_P)\).
Then \(\bw_{p^i}(\la) \geq \bw_{p^i}(\mu)\) for all \(i \geq 1\) by \Cref{thm:determine_weights}.
Then \(\chi(1)_p \leq \psi(1)_p\) by \Cref{lemma:ppart_from_weights}.
\end{proof}

\bigskip

%--------------------------------------------------------------------------
\section{Main results for the alternating group}\label{sec:An}

We now turn our attention to alternating groups.
The goal of this section is to prove \Cref{thm:main-alt-explicit,thm:main-alt-compare-pparts}.
We quickly deal with the case of \(p\) odd, before detailing the explicit method of (almost) determining the \(2\)-height, defect and \(2\)-part of degree of an irreducible character from its vanishing set in \Cref{thm:alt-determine-weights,cor:alt-determine-data}.

We first observe that the vanishing set of an irreducible character of \(A_n\) is the same as the vanishing set of any one of its covering characters of \(S_n\) restricted to even permutations.

\begin{proposition}\label{prop:An-vanishing}
	Let $n\in\N$. Suppose $\chi\in\Irr(S_n)$ and $\eta\in\Irr(A_n)$ satisfy $[\chi\down_{A_n},\eta]\neq 0$. 
    Then for any \(g \in A_n\), we have $\chi(g)=0$ if and only if $\eta(g)=0$.
    That is, \(\Van(\eta) = \Van(\chi) \cap A_n\).
\end{proposition}

\begin{proof}
	The assertion is clear if $\chi\down_{A_n}\in\Irr(A_n)$. Otherwise, $\chi=\chi^\lambda$ where $\lambda=\lambda'\in\cP(n)$ with $n>1$, and $\eta\in\{\eta^{\lambda+}, \eta^{\lambda-}\}$. 
	In this case, the statement follows immediately from the description of the character values of $\chi^\lambda$ and of $\eta$ given in \cite[Corollary 2.4.8 and Theorem 2.5.13]{JK}.
\end{proof}

The cases of \(p\) odd in the statements of \Cref{thm:main-alt-explicit} and \Cref{thm:main-alt-compare-pparts}
then follow immediately from the corresponding statements for the symmetric group: for \(p\) odd, a Sylow \(p\)-subgroup of \(A_n\) is precisely a Sylow \(p\)-subgroup of \(S_n\), and the \(p\)-part of degree, defect of a \(p\)-block and \(p\)-height are the same for $\eta\in\Irr(A_n)$ as for $\chi\in\Irr(S_n)$ covering \(\eta\).

We now proceed with the case \(p=2\).
As for the symmetric group, our strategy will be to determine the \(2\)-power weights -- though here we will only be able to identify them up to two possibilities.

We would like to use \Cref{thm:determine_weights}, but the elements used in that theorem may not lie in \(A_n\).
Nevertheless, in the following proposition we identify a different non-zero character value that will get us most of the way towards our goal.
We remark that, compared to \Cref{thm:determine_weights}, the proof has become considerably more difficult, despite the small change to the hypothesis and the restriction to the case \(p=2\). 

\begin{proposition}\label{prop:weight_minus_one}
	Let $l,n\in\N$ and $\lambda\in\cP(n)$. Suppose $w=\bw_{2^{l}}(\lambda)$ is odd. Let $\pi\in S_n$ be a product of $(w-1)$ pairwise disjoint $2^{l}$-cycles. Then $\chi^\lambda(\pi)\ne 0$.
\end{proposition}

\begin{remark}
    The assumption that $w$ is odd is necessary here.
    For example, let $n=4$, $l=1$ and $\lambda = (2,2)$.
    Then $w = \bw_2(\lambda) = 2$ is even, $\pi$ is a single transposition and $\chi^\lambda(\pi) = 0$.
\end{remark}

\begin{proof}
	Let the $2^{l}$-quotient of $\lambda$ be denoted by
	\[ Q_{2^{l}}(\lambda) = \big( \lambda^{(0)}, \lambda^{(1)}, \ldots, \lambda^{(2^{l}-1)}\big). \]
	For $i \in \{0, 1, \ldots, 2^{l}-1\}$, let $\mu^{(i)}$ be the partition with the same $2^{l}$-core as $\lambda$, but with $2^{l}$-quotient given by
	\[ Q_{2^{l}}(\mu^{(i)}) = \big( \emptyset, \ldots, \emptyset, (1), \emptyset, \ldots, \emptyset \big) \]
	where $(1)$ appears in the $i$th position.
	Observe that if we remove $(w-1)$-many $2^{l}$-hooks from $\lambda$, we necessarily arrive at $\mu^{(i)}$ for some $i$.
    Recalling the definition of paths from \cref{not:path}, we set \(\varepsilon_i = \sgn(\la \setminus \mu^{(i)})\) and use the \MN rule (\cref{thm:MN-reformulated}) to find
	\begin{align}\label{eq:mn_for_weight_minus_one}
		\chi^\lambda(\pi)
		&= \sum_{i\in\{0,1,\dotsc,2^{l}-1\}} \varepsilon_{i} \cdot \underbrace{\#\{\text{paths }\lambda\xrightarrow{2^{l}}\mu^{(i)} \}}_{=:\sP(i)} \cdot \chi^{\mu^{(i)}}(1).
	\end{align}
	
	Fix $i\in\{0,1,\dotsc,2^{l}-1\}$. We calculate $\sP(i)$, the number of different paths $P:\lambda\xrightarrow{2^{l}}\mu^{(i)}$. Clearly $\sP(i)=0$ if $\lambda^{(i)}=\emptyset$.
	Otherwise, observe that such a path $P$ is equivalent by \cref{prop:quotient-hooks} to a way of removing $(w-1)$-many boxes in total from the (Young diagrams of the) components of $Q_{2^{l}}(\la)$, one at a time leaving partitions at each step, until we arrive at $Q_{2^{l}}(\mu^{(i)})$. Letting $w_{j}:=|\lambda^{(j)}|$, we therefore have
	\begin{equation}\label{eq:P(i)}
		\sP(i) = \binom{w-1}{w_{0},\dotsc,w_{i-1},w_i-1,w_{i+1},\ldots,w_{2^{l}-1}}\cdot \prod_{j\in\{0,1,\ldots,2^{l}-1\}}\chi^{\lambda^{(j)}}(1),
	\end{equation}
	interpreting this as zero if $w_{i}=0$.
	This is because the first multinomial term describes the sequence of components where each successive box is removed from, whilst the second product term describes for each component $\lambda^{(j)}$ the number of paths $\lambda^{(j)}\xrightarrow{1}\emptyset$ (i.e.~removing boxes one by one whilst leaving a partition at each step), which is precisely the number of standard Young tableaux of shape $\lambda^{(j)}$ and is equal to $\chi^{\lambda^{(j)}}(1)$. Indeed, for each $j\ne i$ we must remove $w_{j}$ boxes altogether from the component $\lambda^{(j)}$, whilst from the component $\lambda^{(i)}$ itself we only remove $w_{i}-1$ boxes, but there is a natural bijection between paths $\lambda^{(i)}\xrightarrow{1}\emptyset$ and paths $\lambda^{(i)}\xrightarrow{1}(1)$ since $\lambda^{(i)}\ne \emptyset$.
	
	Expanding the multinomial term in \eqref{eq:P(i)} and rearranging, observe that $\sP(i)$ equals $w_i$ multiplied by a non-zero expression that does not depend on $i$, namely
	\[ \sP(i) = w_i \cdot \frac{(w-1)!}{w_0! \cdots w_i! \cdots w_{2^{l}-1}!} \prod_{j\in\{0,1,\dotsc,2^{l}-1\}}\chi^{\lambda^{(j)}}(1). \]

    We now turn our attention to the \(\chi^{\mu^{(i)}}(1)\) terms in the expression for \(\chi^\la(\pi)\) in \eqref{eq:mn_for_weight_minus_one}.
    We first claim that the \(\mu^{(i)}\) all have the same \(2\)-power weights.
    Indeed, the $\mu^{(i)}$ all have the same $2^{l}$-core (equal to the $2^{l}$-core of $\lambda$) and have $2^{l}$-quotients which differ only by permuting the components.
    Then we can use \Cref{prop:iterated_cores_and_quotients} to express the weights \(\bw_{2^r}(\mu^{(i)}) = \abs{Q_{2^r}(\mu^{(i)})}\) in terms of the \(2^l\)-cores and -quotients:
    for \(r \geq l\), we have \(\bw_{2^r}(\mu^{(i)}) = \abs{Q_{2^{r-l}}(Q_{2^l}(\mu^{(i)}))}\),
    while for \(r < l\): %we have
    \begin{align*}
    \bw_{2^r}(\mu^{(i)})
        &= \abs{C_{2^{l-r}}(Q_{2^r}(\mu^{(i)}))} + 2^{l-r}\abs{Q_{2^{l-r}}(Q_{2^r}(\mu^{(i)}))} \\
        &= \abs{Q_{2^r}(C_{2^l}(\mu^{(i)}))} + 2^{l-r}\abs{Q_{2^l}(\mu^{(i)})}.
    \end{align*}
    Hence $\bw_{2^r}(\mu^{(i)})$ is independent of $i$, for all $r$.
	Thus, by \cref{lemma:ppart_from_weights}, the $2$-parts of the degrees $\chi^{\mu^{(i)}}(1)$ are all equal.
	Let $c\in\N_0$ be such that for all $i$ we have $\chi^{\mu^{(i)}}(1) = 2^c m_i$ where $m_i$ is odd.
    
    Substituting into \eqref{eq:mn_for_weight_minus_one}, we have
	\begin{align*}
	    \chi^\lambda(\pi)
	    &= 2^c \frac{(w-1)!}{w_0! \cdots w_i! \cdots w_{2^{l}-1}!} \prod_{j\in\{0,1,\dotsc,2^{l}-1\}}\chi^{\lambda^{(j)}}(1)  \sum_{i\in\{0,1,\dotsc,2^{l}-1\}} \eps_i w_i m_i.
	\end{align*}
	Therefore, in order to conclude that $\chi^\lambda(\pi)\ne 0$, it suffices to show that the integer $\sum_{i} \eps_i w_i m_i$ is non-zero.
    We do this by showing that it is odd. Indeed, we have $\eps_i \equiv m_i \equiv 1 \pmod{2}$ and $w \equiv 1 \pmod{2}$ by hypothesis, so that
	\[ \sum_{i} \eps_i w_i m_i \equiv \sum_i w_i = w \equiv 1 \pmod{2}. \qedhere \]
	% as required.
\end{proof}

Equipped with \Cref{prop:weight_minus_one}, we are able to almost uniquely determine the \(2\)-power weights as follows.

\begin{definition}
\label{def:weight-estimates}
	Let \(n \in \N\) and let \(\eta \in \Irr(A_n)\). Let \(Q\) be a Sylow \(2\)-subgroup of \(A_n\). For each \(i \geq 1\) we define the non-negative numbers \(\auxw_{2^i}\) and \(\hatw_{2^i}\) from the vanishing set of \(\eta\) on \(Q\) as follows:
\begin{align*}
\auxw_{2^i} &:=
    \max\left\{ \sum_{j\ge i} 2^{j-i}b_j \ \middle\vert\ \ \parbox{154pt}{\(\eta(\sigma)\ne 0\) for some \(\sigma\in Q\) of cycle type \(\big( \dotsc, (2^2)^{b_2}, (2^1)^{b_1}, (2^0)^{b_0} \big)\) }\ \right\}
\intertext{and}
\hatw_{2^i} &:=
	\begin{cases}
		\auxw_{2^i} + 1 & \text{ if \(i=1\) and \(n-2\auxw_{2}\) is not a triangular number}, \\
		\auxw_{2^i} & \text{ otherwise.}
	\end{cases}
\end{align*}
\end{definition}

\begin{remark}
The number \(\hatw_{2^i}\) is a natural way to estimate the weight \(\bw_{2^i}(\la)\) of a partition \(\la\) labelling an irreducible character of \(S_n\) covering \(\eta\),
because \(\auxw_{2^i}\) counts the maximum number of \(2^i\)-hooks that are removed from \(\la\) when calculating \(\chi^\la(\sigma)\) for some \(\sigma \in Q\) with the \MN rule.
When \(i=1\), \(\hatw_2\) offers an improvement over \(\auxw_2\) by considering the fact that the \(2\)-core of a partition is necessarily of size a triangular number.
Nevertheless, for simplicity, the reader may prefer to ignore this improvement, and instead use the definition of \(\auxw_{2^i}\) as the definition of \(\hatw_{2^i}\).
The statements of \Cref{lemma:weight-estimates,thm:alt-determine-weights,cor:alt-determine-data} below still hold as written with either definition (and by the same proofs, with only one trivial change in the first); however, the conditions for the estimates to be known to be correct are met less frequently when the unimproved definition is used.
\end{remark}

\begin{lemma}
\label{lemma:weight-estimates}
Let \(n \in {\N}\) and let \(\eta \in \Irr(A_n)\). Let \(\la \in \cP(n)\) be such that \(\eta\) is covered by \(\chi^\la\).
For each $i\ge 1$, let \(\hatw_{2^i}\) be defined as in \Cref{def:weight-estimates}.
Then:
\begin{enumerate}[(i)]
    \item\label{item:upper-bound}
        \(\hatw_{2^i} \leq \bw_{2^i}(\la)\) for all \(i \geq 1\);
    \item \label{item:even-weight-lemma}
        if \(i \geq 1\) is such that \(\bw_{2^i}(\la)\) is even, then \(\hatw_{2^i} = \bw_{2^i}\);
    \item \label{item:odd-weight-lemma}
        if \(i \geq 1\) is such that \(\bw_{2^i}(\la)\) is odd, then \(\hatw_{2^i} \geq \bw_{2^i} - 1\);
	\item\label{item:two-odd-weights}
	    if distinct \(i,j \geq 1\) are such that \(\bw_{2^i}(\la), \bw_{2^j}(\la)\) are odd, then \(\hatw_{2^i} = \bw_{2^i}(\la)\) and \(\hatw_{2^j} = \bw_{2^j}(\la)\).
\end{enumerate}
\end{lemma}

\begin{proof}
Since \(\auxw_{2^i}\) counts the maximum number of \(2^i\)-hooks that are removed from \(\la\) when calculating \(\chi^\la(\sigma)\) for some \(\sigma \in Q\) with the \MN rule, part \Cref{item:upper-bound} is clear for \(i > 1\).
For \(i=1\), we know \(n-2\bw_2(\la)\) is the size of the \(2\)-core of \(\la\) and hence necessarily a triangular number; thus if \(n - 2\auxw_2\) is not a triangular number then \(\auxw_2 < \bw_2(\la)\) and hence \(\hatw_2 \leq \bw_2(\la)\), as required.

If \(\bw_{2^i}(\la)\) is even, a product of \(\bw_{2^i}(\la)\) disjoint \(2^i\)-cycles lies in \(A_n\) and \(\eta\) is nonzero on this product by \Cref{thm:determine_weights}, yielding part \Cref{item:even-weight-lemma}.
If \(\bw_{2^i}(\la)\) is odd, a product of \((\bw_{2^i}(\la)-1)\) disjoint \(2^i\)-cycles lies in \(A_n\) and \(\eta\) is nonzero on this product by \Cref{prop:weight_minus_one}, yielding part \Cref{item:odd-weight-lemma}.

For part \Cref{item:two-odd-weights}, without loss of generality suppose \(i > j\).
Let \(\rho\) be a product of \(\bw_{2^i}(\la)\) many \(2^i\)-cycles and let \(\pi\) be a product of \((\bw_{2^j}(\la) - 2^{i-j}\bw_{2^i}(\la))\) many \(2^j\)-cycles, all disjoint.
(Indeed, we can choose $\pi$ and $\rho$ to be disjoint since the number of moved points of $\rho$ is $n-|C_{2^i}(\lambda)|$, and of $\pi$ is $|C_{2^i}(\lambda)|-|C_{2^j}(\lambda)|$.)
Observe that \(\pi\rho \in A_n\), and that \(\eta(\pi\rho) \neq 0\): we have that \(\chi^\la(\pi\rho)\) is a non-zero multiple of \(\chi^{C_{2^i}(\la)}(\pi)\) by \Cref{lemma:unique_partition_implies_non-zero}, and \(C_{2^i}(\la)\) has \(2^j\)-weight \(\bw_{2^j}(C_{2^i}(\la)) = \bw_{2^j}(\la) - 2^{i-j}\bw_{2^i}(\la)\) (as can be verified using \Cref{prop:iterated_cores_and_quotients}\Cref{item:corequot_vs_quotcore}), so \(\chi^{C_{2^i}(\la)}(\pi) \neq 0\) by \Cref{lemma:unique_partition_implies_non-zero} again.
Thus the maximum in \Cref{def:weight-estimates} is attained by \(\pi\rho\).
\end{proof}

\begin{theorem}
\label{thm:alt-determine-weights}
Let \(n \in {\N}\) and let \(\eta \in \Irr(A_n)\). Let \(\la \in \cP(n)\) be such that \(\eta\) is covered by \(\chi^\la\).
For each $i\ge 1$, let \(\hatw_{2^i}\) be defined as in \Cref{def:weight-estimates}.
Then
\[
	\bw_{2^i}(\la) \in \{\, \hatw_{2^i},\, \hatw_{2^i}+1\,\}.
\]
Furthermore:
\begin{enumerate}[(a)]
	\item\label{item:odd-weight}
	    if \(\hatw_{2^j}\) is odd for any \(j \geq 1\), then \(\bw_{2^i}(\la) = \hatw_{2^i}\) for all \(i \geq 1\);
	\item\label{item:at-most-one-error}
	    \(\bw_{2^j}(\la) = \hatw_{2^j}+1\) for at most one \(j \geq 1\).
\end{enumerate}
\end{theorem}

\begin{proof}
The main statement follows immediately from the first three parts of \Cref{lemma:weight-estimates}.
	
For part \Cref{item:odd-weight},
suppose there exists \(j\) such that \(\hatw_{2^j}\) is odd.
From \Cref{lemma:weight-estimates}\Cref{item:even-weight-lemma},\Cref{item:odd-weight-lemma}, we deduce that \(\bw_{2^j}(\la)\) must be odd and equal to \(\hatw_{2^j}\).
Then if \(i \neq j\) is such that \(\bw_{2^i}(\la)\) is odd, by \Cref{lemma:weight-estimates}\Cref{item:two-odd-weights} we have that \(\bw_{2^i}(\la) = \hatw_{2^i}\) (while if \(\bw_{2^i}(\la)\) is even then we have \(\bw_{2^i}(\la) = \hatw_{2^i}\) already by \Cref{lemma:weight-estimates}\Cref{item:even-weight-lemma}).

Finally we show \Cref{item:at-most-one-error}.
By part \Cref{lemma:weight-estimates}\Cref{item:even-weight-lemma}, if the \(2\)-power weights \(\bw_{2^i}(\la)\) are even for all $i\ge 1$, then \(\bw_{2^i}(\la) = \hatw_{2^i}\) for all $i\ge 1$.
Meanwhile if there are at least two odd \(2\)-power weights, then we have that \(\bw_{2^j}(\la) = \hatw_{2^j}\) for all \(j\ge 1\) with \(\bw_{2^j}(\la)\) odd by \Cref{lemma:weight-estimates}\Cref{item:two-odd-weights}, and hence \(\bw_{2^j}(\la) = \hatw_{2^j}\) for all \(j\ge 1\) using \Cref{lemma:weight-estimates}\Cref{item:even-weight-lemma} again.
The only remaining situation is having exactly one odd \(2\)-power weight, in which only for the \(j\) such that \(\bw_{2^j}(\la)\) is odd is it possible to have \(\bw_{2^j}(\la) = \hatw_{2^j}+1\).
\end{proof}

\begin{definition}
	\label{def:alt-data-estimates}
	Let \(n \in \N\) and let \(\eta \in \Irr(A_n)\).
	Let \(m\) be the binary digit sum of \(n\) (that is, \(m = \sum_{i \geq 0} b_i\) where \(n = \sum_{i \geq 0} b_i2^i\) is the binary expansion of \(n\)).
	We define the non-negative numbers \(\hatnu\), \(\hatd\) and \(\hath\) in terms of the numbers \(\hatw_{2^i}\) as follows:
	\begin{enumerate}[(i)]
	    \item \(\hatnu = n-m - \sum_{i \geq 1} \hatw_{2^i}\); \vspace{3pt}
	    \item \(\hatd = \max \{\nu_2( (2\hatw_2)! ) - 1,\, 0\}\); \vspace{5pt}
	    \item \(\hath = \hatnu + \hatd - (n-m-1)\).
	\end{enumerate}
\end{definition}

\begin{corollary}
\label{cor:alt-determine-data}
	Let \(n \in \N\), \(n \geq 2\), and let \(\eta \in \Irr(A_n)\). Let $\lambda\in\cP(n)$ be such that $\eta$ is covered by $\chi^\lambda$.
	The \(2\)-part of the degree of \(\eta\), the defect \(d\) of the \(2\)-block in which \(\eta\) lies, and the \(2\)-height of \(\eta\) can all be determined up to two possibilities as follows:
	\begin{align*}
	    \eta(1)_2 &=
	            \begin{cases}
	            2^{\hatnu} & \text{ if \(\hatw_{2^i} = \bw_{2^i}(\la) \) for all \(i \geq 1\) and \(\la\) is not self-conjugate,} \\
	            2^{\hatnu-1} & \text{ otherwise;}
	            \end{cases} \\
	    d &=
	            \begin{cases}
	            \hatd & \text{ if \(\hatw_{2} = \bw_{2}(\la)\) or \(\hatw_2 = 0\),} \\
	            \hatd+1 & \text{ otherwise;}
	            \end{cases} \\
	    h_2(\eta) &=
	            \begin{cases}
	            \hath & \text{ if \(\hatw_2 > 0\), \(\hatw_{2^i} = \bw_{2^i}(\la) \) for all \(i \geq 2\), and \(\la\) is not self-conjugate,} \\
	            \hath-1 & \text{ otherwise.}
	            \end{cases}
	\end{align*}
\end{corollary}

\begin{remark}
	The conditions to distinguish the pairs of options in \Cref{cor:alt-determine-data} 
 (that is, the \(2\)-power weights and the self-conjugacy of the labelling partition) cannot in general be checked using vanishing on even permutations of \(2\)-power order.
	
However, there are many situations when the vanishing set can tell us which option holds.
For example, if there exists \(j \geq 1\) such that \(\hatw_{2^j}\) is odd, then $\lambda$ cannot be self-conjugate, and so (if also \(\hatw_2 > 0\)) the first option holds for each of the three statistics by \Cref{thm:alt-determine-weights}\Cref{item:odd-weight}.
As another example, we claim that if \(n-2\hatw_2 \neq 3\) or if \(n \not\equiv 3 \pmod{4}\), then \(\hatw_2 = \bw_2(\la)\) and hence \(d = \hatd\).
Indeed, \(\hatw_2 = \bw_2(\la) -1\) occurs only if \(\bw_2(\la)\) is odd, \(\auxw_2 = \hatw_2\) is even, and \(n-2\hatw_2\) is a triangular number.
But then \(n-2(\hatw_2+1)\) is also a triangular number (being the size of the \(2\)-core of \(\la\)), and \(1\) and \(3\) are the only pair of triangular numbers differing by \(2\), so \(n-2\hatw_2 = 3\) and hence \(n \equiv 3 \pmod{4}\).
\end{remark}

\begin{proof}[Proof of \Cref{cor:alt-determine-data}]
We have \(\nu_2(\chi^\la(1)) = n-m - \sum_{i \geq 1} \bw_{2^i}(\la)\) by \Cref{lemma:ppart_from_weights}.
Then by \Cref{thm:alt-determine-weights}, we have \(\nu_2(\chi^\la(1)) \in \{ \hatnu-1, \hatnu\}\) with \(\nu_2(\chi^\la(1)) = \hatnu\) if and only if \(\hatw_{2^i} = \bw_{2^i}(\la) \) for all \(i \geq 1\).
If \(\la\) is self-conjugate, then \(\eta(1) = \frac12 \chi^\la(1)\) and necessarily $\bw_{2^i}(\lambda)$ is even for all $i\ge 1$ (which implies \(\hatw_{2^i} = \bw_{2^i}(\la) \) for all \(i \geq 1\) by \Cref{lemma:weight-estimates}\Cref{item:even-weight-lemma}), so \(\nu_2(\eta(1)) = \nu_2(\chi^\la(1)) - 1 = \hatnu-1\).
If \(\la\) is not self-conjugate, then \(\eta(1) = \chi^\la(1)\) and the result for \(\eta(1)_2\) follows.
	
Next, recall the block of \(A_n\) containing \(\eta\) has defect \(d = \max\{\nu_2((2\bw_2(\la))!) - 1,0\}\) (see \Cref{subsec:blocks}).
Thus if \(\hatw_2 \in \{ \bw_2(\la), 0\}\) then $d=\hatd$ as claimed.
% (bearing in mind that if \(\hatw_2 = 0\) then \(\bw_2(\la) = 1\)).
Otherwise, $\bw_2(\lambda)=\hatw_2+1$,
which can happen only if $\bw_2(\lambda)$ is odd by \Cref{lemma:weight-estimates}\Cref{item:even-weight-lemma}.
Hence
\begin{align*}
\nu_2( (2\bw_2(\la))! )
    &= \nu_2( (2\hatw_2)!) + \nu_2(2\bw_2(\la){-}1) + \nu_2( 2\bw_2(\la) ) \\
    &= \nu_2( (2\hatw_2)!) + 1
\end{align*}
and so  \(d = \max\{\, \nu_2( (2\bw_2(\la))! ) - 1,0 \,\} = \nu_2( (2\hatw_2)! )\).
On the other hand, \(\hatd = \nu_2( (2\hatw_2)! ) - 1\) for \(\hatw_2 > 0\).
Thus $d = \hatd+1$ is as claimed.

Finally, \(h_2(\eta) = \nu_2(\eta(1)) + d - (n-m-1)\) (see \Cref{subsec:blocks}, and use Legendre's formula to find the order of a Sylow \(2\)-subgroup of \(A_n\)),
so it suffices to compare the values of the previous two statistics under the claimed conditions. 
If \(\hatw_2 > 0\), \(\hatw_{2^i} = \bw_{2^i}(\la) \) for all \(i \geq 2\), and \(\la\) is not self-conjugate, then either: \(\hatw_2 = \bw_2(\la)\) in which case \(\nu_2(\eta(1)) = \hatnu\) and \(d = \hatd\) so \(h_2(\eta) = \hath\); or \(\hatw_2 \neq \bw_2(\la)\) in which case \(\nu_2(\eta(1)) = \hatnu-1\) and \(d = \hatd+1\) so again \(h_2(\eta) = \hath\).
Otherwise:

\begin{itemize}[leftmargin=*,topsep=0pt]
    \item
    If \(\hatw_2 = 0\): either \(\bw_2(\la) = 0\) and so \(\la\) is self-conjugate, or \(\bw_2(\la)=1\) and so \(\hatw_2 \neq \bw_2\); in either case, \(\nu_2(\eta(1)) = \hatnu-1\). Also, \(d = \hatd\), and so \(h_2(\eta) = \hath-1\).
\item
If there exists \(i \geq 2\) such that \(\hatw_{2^i} \neq \bw_{2^i}(\la)\): then \(\nu_2(\eta(1)) = \hatnu-1\).
Also, \(\hatw_2 = \bw_2(\la)\) by \Cref{thm:alt-determine-weights}\Cref{item:at-most-one-error}, and so \(d= \hatd\). Thus \(h_2(\eta) = \hath-1\).
\item
If \(\la\) is self-conjugate: then \(\nu_2(\eta(1)) = \hatnu-1\). Also, $\bw_{2}(\lambda)$ is even and thus $\hatw_{2}=\bw_{2}(\lambda)$ by \Cref{lemma:weight-estimates}\Cref{item:even-weight-lemma}, and so \(d= \hatd\). Thus \(h_2(\eta) = \hath-1\). \qedhere
\end{itemize}
\end{proof}

\Cref{cor:alt-determine-data} immediately yields the \(p=2\) case of \Cref{thm:main-alt-explicit}, and allows us to deduce \Cref{thm:main-alt-compare-pparts} as follows.
This completes the proofs of our main theorems.

\begin{proof}
[Proof of \(p=2\) case of \Cref{thm:main-alt-compare-pparts}]
Suppose \(\Van(\eta\down_Q) \subseteq \Van(\theta\down_Q)\).
Then \(\hatw_{2^i}(\eta) \geq \hatw_{2^i}(\theta)\) for all \(i \geq 1\) by definition, and hence \(\hatnu(\eta) \leq \hatnu(\theta)\).
Then \(\eta(1)_2 \leq 2^{\hatnu(\eta)} \leq 2^{\hatnu(\theta)} \leq 2 \cdot \theta(1)_2\) by \Cref{cor:alt-determine-data}.
\end{proof}

\bigskip

%--------------------------------------------------------------------------
\section{Examples of vanishing behaviour}\label{sec:examples}

In this final section, we present a number of examples to show that our main theorems cannot be strengthened further in terms of characterising degrees using vanishing behaviour.

We will construct examples in terms of their \emph{\(e\)-core towers} (see \cite[Section 6]{OlssonBook}).
These towers are ways to express partitions by iteratively taking cores and quotients: given a partition \(\la\), its \(e\)-core tower is the sequence
\[ \big(T^C_e(\la)_k\big)_{k=0}^{\infty}, \]
where the \(k\)th term, referred to as the \emph{\(k\)th layer} (or \emph{\(k\)th row}), is the \(e^k\)-tuple of partitions
\[
    T^C_e(\la)_k = C_e\left(\underbrace{Q_e(Q_e(\cdots Q_e(}_{k}\la) \cdots )) \right).
\]

Equivalently, the tower can be constructed by iteratively replacing a partition with its \(e\)-core and placing its \(e\)-quotient in the next layer.
Following this construction, we draw a tower as rooted \(e\)-ary tree by placing an edge between each \(e\)-core and the components of the \(e\)-quotient at each step.
This process is illustrated below in the case \(e=2\) and \(\la = (5,3,3,3,1)\).

\vspace{-5pt}
{
\Yvcentermath0
\Yboxdim{9pt}
\[
% \ 
% \gyoung(;;!\YcB6;2!\wht;,;6;4;,;;!\YcB;2!\wht,42;,;)
\raisebox{-4pt}{\yng(5,3,3,3,1)}
\qquad\raisebox{-1cm}{\(\rightsquigarrow\)}\qquad
\begin{forest}
    tower
    [{\yng(1)}, baseline
        [{\yng(2,1)}]
        [{\yng(2,2)}]
    ]
\end{forest}
\qquad\raisebox{-1cm}{\(\rightsquigarrow\)}\qquad
\begin{forest}
    tower
    [{\yng(1)}, baseline, name=layerzero
        [{\yng(2,1)},
            [\text{\Large\(\emptyset\)}]
            [\text{\Large\(\emptyset\)}]
        ]
        [\text{\Large\(\emptyset\)}, name=layerone
            [{\yng(1)}]
            [{\yng(1)}, name=layertwo] 
        ] %{\node at (.east) [anchor=west,align=center,right=4ex]{\(T^C_2(\la)_1\)};}
    ] %{\node at (.east) [anchor=west,align=center,right=4ex]{\(T^C_2(\la)_0\)};}
\node at (3.5, |- layerzero) {\(= \ T^C_2(\la)_0\)};
\node at (3.5, |- layerone) {\( = \ T^C_2(\la)_1\)};
\node at (3.5, |- layertwo) {\( = \ T^C_2(\la)_2\)};
\end{forest}
\]
}
\vspace{0pt} % for some reason this \vspace{0pt} helps the spacing

Repeated applications of \cite[Proposition 3.7]{OlssonBook} show that every partition is uniquely determined by its $e$-core tower, and moreover, we have that $n=\sum_{k=0}^{\infty} |T_e^C(\lambda)_k|e^k$.

In addition, the $e$-core tower $T_e^C(\lambda')$ satisfies $T_e^C(\lambda')_0=(C_e(\lambda)')$, and for each $k\in\N$, the layer $T_e^C(\lambda')_k$ is obtained from $T_e^C(\lambda)_k$ by taking the conjugate of each partition and reversing the sequence of partitions. In particular, if $\lambda$ is self-conjugate (meaning $\lambda=\lambda'$) and $e$ is even, then $|T_e^C(\lambda)_k|$ must be even for all $k\in\N$.

Knowing the \(e\)-core tower layer sizes is equivalent to knowing the \(e\)-power weights. Given $\lambda$ in its \(p\)-core tower form, it is easy to calculate $\chi^\lambda(1)_p$, as shown below.

\begin{lemma}\label{prop:weights-to-tower-sizes}
	Let $n,k\in\N$ and $e\in\N_{\ge 2}$. Let $\lambda\in\cP(n)$. The size of the $k$th layer of the $e$-core tower of $\lambda$ is given in terms of the \(e\)-power weights of $\lambda$ by
	\[ \abs{T^C_e(\lambda)_k} = \bw_{e^k}(\lambda) - e\bw_{e^{k+1}}(\lambda). \]
    Let $p$ be a prime and suppose $n=\sum_{r\ge 0}b_rp^r$ is the $p$-adic expansion of $n$.
    Then
	\[ \nu_p(\chi^\lambda(1)) = \frac{\sum_{r\ge 0}|T^C_p(\lambda)_r| - \sum_{r\ge 0}b_r}{p-1}.\]
\end{lemma}

\begin{proof}
Using the definition of the \(k\)th layer given above and \Cref{prop:iterated_cores_and_quotients}, we have
\begin{align*}
\abs{T^C_e(\la)_k}
    &= \abs{C_e(Q_{e^k}(\la))} \\
    &= \abs{Q_{e^k}(\la)} - e\abs{Q_e(Q_{e^k}(\la))} \\
    &= \abs{Q_{e^k}(\la)} - e\abs{Q_{e^{k+1}}(\la)} \\
    &= \bw_{e^k}(\la) - e\bw_{e^{k+1}}(\la)).
\end{align*}
Summing over all \(k\) we find \(\sum_{k \geq 0} |T^C_e(\la)_k| = n - (e-1)\sum_{i \geq 1} \bw_{e^i}(\la)\), and the expression for \(\nu_p(\chi^\lambda(1))\) follows from \Cref{lemma:ppart_from_weights}.
\end{proof}

\medskip

\subsection{Converses of main theorems fail}

We give two families of examples to show that the vanishing set $\Van(\chi^\lambda\down_P)$ of an irreducible character $\chi^\lambda\in\Irr(S_n)$ is \emph{not} determined by the $p$-part of its degree or its height, or by the \(p\)-power weights of the corresponding partition $\lambda$.

\begin{example}[Converse of \Cref{thm:main-sym-compare-pparts} fails for every odd prime]\label{eg:converse_fails}
    Let $p$ be any odd prime, let $n=p^2$, and let $\lambda = (p(p-1),p-1,1)$ and $\mu = (p(p-1),p)$.
    That is, $\lambda$ and $\mu$ are the partitions with $p$-core towers
	
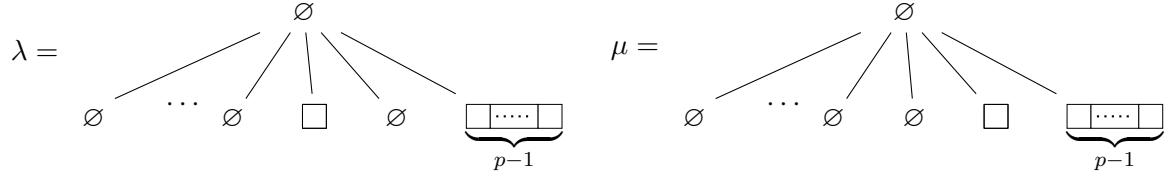
\[
	\begin{forest}
	for tree={l=1.4cm, s sep=0.5cm}
	[$\emptyset$,minimum width=1.3cm
	    [$\emptyset$,child anchor=north east]
	    [$\mathclap{\cdots}$,l=1.2cm, before computing xy={s=-1.6cm}, no edge]
	    [$\emptyset$]
	    [{\yng(1)}] 
	    [$\emptyset$] 
	    [$\underbrace{{\gyoung(;_2\hdts;)}}_{p-1}$,child anchor=north west]
	]
	\node at (current bounding box.west)[left=0.3ex,yshift=0.5cm]
	 {$\lambda = $};
	\end{forest}
	\quad
	\begin{forest}
	for tree={l=1.4cm,s sep=0.5cm}
	[$\emptyset$,minimum width=1.3cm
	    [$\emptyset$,child anchor=north east]
	    [$\mathclap{\cdots}$,l=1.2cm, before computing xy={s=-1.6cm}, no edge]
	    [$\emptyset$]
	    [$\emptyset$] 
	    [{\yng(1)}] 
	    [$\underbrace{{\gyoung(;_2\hdts;)}}_{p-1}$,child anchor=north west] 
	]
	\node at (current bounding box.west)[left=0.3ex,yshift=0.5cm]
	 {$\mu = $};
	\end{forest}
	\]
	These towers differ only by permuting two entries in the first layer $T_p^C(-)_1$, and so have the same layer sizes.
    Hence \(\la\) and \(\mu\) have the same \(p\)-power weights and $\chi^\lambda(1)_p = \chi^\mu(1)_p$
    by \cref{prop:weights-to-tower-sizes}, and hence also $\chi^\lambda$ and $\chi^\mu$ have equal heights since they lie in the same $p$-blocks (because $C_p(\lambda)=C_p(\mu)$). 
	However, $\chi^\mu$ vanishes on products of $(p-1)$ many disjoint $p$-cycles whereas $\chi^\lambda$ does not: writing $g$ for such an element, the \MN rule and the hook length formula give
	\begin{align*}
    	    \chi^\lambda(g) &= -(p-1) \chi^{(p)}(1) - \chi^{(p-2,1,1)}(1) = - \binom{p}{2} \\
		\intertext{but}
	    \chi^\mu(g) &= \phantom{-}(p-1)\chi^{(p)}(1) - \chi^{(p-1,1)}(1) = 0. %\qedhere
	\end{align*}
Thus \(\Van(\chi^\la\down_P) \neq \Van(\chi^\mu\down_P)\).
\end{example}

\begin{example}[Converse of \Cref{thm:main-sym-compare-pparts} fails infinitely often for \(p=2\)]\label{eg:converse_fails_p=2}
Let \(p=2\), let \(r \geq 2\), let \(T_r = \binom{r+1}{2}\) be the \(r\)th triangular number, and let \(\kappa_r\) denote the unique \(2\)-core partition of \(T_r\), that is, $\kappa_r=(r,r-1,\dotsc,2,1)$.
Let \(n = T_r + 8\) and let \(\la\) and \(\mu\) be the partitions with \(2\)-core towers
	
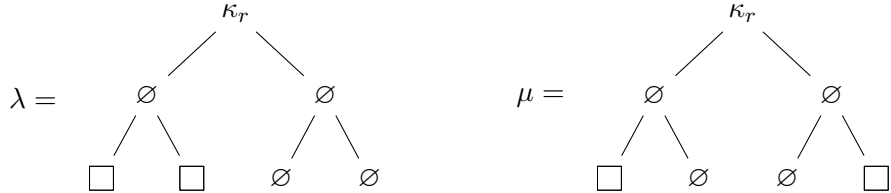
\[
	\begin{forest}
		tower
		[\kappa_r,
		    [\emptyset,
		        [{\yng(1)}]
		        [{\yng(1)}]
		    ]
		    [\emptyset,
		        [\emptyset]
		        [\emptyset]
		    ] 
		]
		\node at (current bounding box.west)[left=1ex]
		 {$\lambda = $};
		\end{forest}
		\qquad\qquad
		\begin{forest}
		tower
		[\kappa_r,
		    [\emptyset,
		        [{\yng(1)}]
		        [\emptyset]
		    ]
		    [\emptyset,
		        [\emptyset]
		        [{\yng(1)}]
		    ] 
		]
		\node at (current bounding box.west)[left=1ex]
		 {$\mu = $};
	\end{forest}
	\]
    	These towers differ only by permuting two entries in the second layer $T_2^C(-)_2$, and so have the same layer sizes. Hence \(\la\) and \(\mu\) have the same \(2\)-power weights and $\chi^\lambda(1)_2 = \chi^\mu(1)_2$ %and $\chi^\mu$ have equal $2$-adic degree valuations
        by \cref{prop:weights-to-tower-sizes}, and hence also $\chi^\lambda$ and $\chi^\mu$ have equal heights since they lie in the same $2$-blocks (because $C_2(\lambda)=C_2(\mu)$). 
    However, \(\chi^\la\) is non-zero on transpositions since \(\la\) has a unique \(2\)-hook (and using \Cref{lemma:unique_partition_implies_non-zero}), while \(\chi^\mu\) vanishes on transpositions since \(\mu\) is self-conjugate.
    Thus \(\Van(\chi^\la\down_P) \neq \Van(\chi^\mu\down_P)\).
\end{example}

\medskip

\subsection{Cannot fully determine statistics of characters of \texorpdfstring{\(A_n\) when \(p=2\)}{An when p=2}}

We give several examples to demonstrate that we cannot uniquely determine the \(2\)-part of the degree, the defect of the \(2\)-block, or the \(2\)-height of an irreducible character of \(A_n\) from the vanishing set on a Sylow \(2\)-subgroup \(Q\) of $A_n$.
We present our examples in \cref{tab:Q-vanishing_pairs}; in each case, an example is given as a pair of distinct, non-conjugate partitions \(\la\) and \(\mu\) which label irreducible characters of \(A_n\) with equal vanishing sets on \(Q\), and which have the claimed \(2\)-parts of degrees, defects of \(2\)-blocks, \(2\)-heights, and self-conjugacy.
That they have the claimed data is clear from the \(2\)-core towers of the partitions and \Cref{prop:weights-to-tower-sizes}.
That they have the same vanishing sets is verified by direct computation.
In the cases of \Cref{egtab:equal_data_neither_sc,,egtab:equal_data_both_sc,egtab:equal_data_onesc,egtab:diff_2part_diff_height_onesc}, % empty reference included in list intentionally to force cleveref to list all four items rather than compress to a range
in fact the \(2\)-core can be replaced by any larger \(2\)-core to obtain infinitely many examples with the same properties; the equal vanishing sets for these families can be shown using (sometimes challenging) applications of the \MN rule and the hook length formula.

The significance of these examples is as follows.
Recall that \Cref{cor:alt-determine-data} tells us that, amongst irreducible characters with equal vanishing sets in $Q$, the three pieces of data in consideration differ by at most $1$.
Our examples show that such a pair of characters may agree or disagree on any combination of these three numbers, subject to the relationship between \(2\)-part of degree, defect and height requiring that if any two agree or disagree then the third must agree.
In particular, given two such characters with equal vanishing on \(Q\), additionally specifying one of the three statistics is insufficient to determine either of the other two (so, for example, two irreducible characters lying in the same block with the same vanishing set on \(Q\) may still have different \(2\)-part of degree and \(2\)-height, as in \Cref{egtab:diff_2part_diff_height_neither_sc,egtab:diff_2part_diff_height_onesc}).

Our examples furthermore show that, with one exception, each of the combinations of disagreement can occur whether or not one of the labelling partitions is self-conjugate.
Self-conjugacy of the labelling partition is equivalent to the irreducible character of $S_n$ splitting into a sum of two irreducible characters upon restriction to $A_n$, and hence the \(2\)-part of the degree halving %upon character restriction from \(S_n\) to \(A_n\)
-- and so is a potential cause for disagreements in our statistics, in addition to the uncertainty in the \(2\)-power weights described in \Cref{thm:alt-determine-weights}.
Our examples demonstrate that both causes indeed manifest, and that, in fact, when they occur together they can cancel out (as in \Cref{egtab:equal_data_onesc}).
This also highlights that the vanishing set on \(Q\) is insufficient to detect self-conjugacy of the labelling partition (even if given all three of our statistics, as in \Cref{egtab:equal_data_onesc}).

Note that if we did know that $\eta\in\Irr(A_n)$ was labelled by a self-conjugate partition, then by \Cref{cor:alt-determine-data} the three statistics are uniquely determined from the vanishing set.
Thus, in particular, if two irreducible characters of $A_n$ with equal vanishing on \(Q\) are both labelled by self-conjugate partitions, necessarily they agree on all three statistics.

The one exception to self-conjugacy being possible is the following.

\begin{proposition}
	Let \(n \in \N\). Let $Q\in\Syl_2(A_n)$. Let \(\eta,\theta \in \Irr(A_n)\) with \(\Van(\eta\down_Q) = \Van(\theta\down_Q)\).
	Suppose \(\eta(1)_2 \neq \theta(1)_2\) and that \(\eta\), \(\theta\) lie in distinct blocks.
	Then \(\eta\) and \(\theta\) are labelled by partitions which are not self-conjugate.
\end{proposition}

\begin{proof}
	Let \(\la, \mu \in \cP(n)\) be such that \(\chi^\la\) covers \(\eta\) and \(\chi^\mu\) covers \(\theta\).
	We use the notation of \Cref{sec:An}, with \(\hatw_{2^i}\), \(\hatnu\) and \(\hatd\) constructed from the common vanishing set of \(\eta\) and \(\theta\).
	
	Suppose towards a contradiction that \(\mu\) is self-conjugate.
	Then \(\bw_{2}(\mu)\) is even so \(\hatw_{2} = \bw_{2}(\mu)\) by \cref{lemma:weight-estimates}\Cref{item:even-weight-lemma}.
	Also, by \Cref{cor:alt-determine-data}, we have \(\theta(1)_2 = 2^{\hatnu-1}\), and in order that \(\eta(1)_2 \neq \theta(1)_2\), we must have \(\la\) not self-conjugate and \(\bw_{2^i}(\la) = \hatw_{2^i}\) for all $i\ge 1$.
	But then, by \Cref{cor:alt-determine-data} again, the defects of the blocks containing \(\eta\) and \(\theta\) are both of defect \(\hatd\), and so \(\eta\) and \(\theta\) lie in the same block (since there is a unique \(2\)-core of any given size), a contradiction.
\end{proof}

\newcommand{\exlabelnudge}{21pt}

\newcommand{\partbypartssize}{\tiny}

\newcommand{\OOOnsc}{
	\begin{tabexample}\label{egtab:equal_data_neither_sc}\end{tabexample}
	&
	\begin{forest}
		compact tower
		[\emptyset, % \kappa_r
		[\emptyset, baseline
		[\emptyset]
		[\emptyset]
		]
		[\emptyset,
		[\emptyset]
		[{\yng(1)}]
		] 
		]
		\node at (current bounding box.south)[below=0ex]{\partbypartssize $= (4)$};
	\end{forest}
	&
	\begin{forest}
		compact tower
		[\emptyset, % \kappa_r
		[\emptyset, baseline
		[\emptyset]
		[\emptyset]
		]
		[\emptyset,
		[{\yng(1)}]
		[\emptyset]
		] 
		]
		\node at (current bounding box.south)[below=0ex]{\partbypartssize $= (3,1)$};
	\end{forest} 
}

\newcommand{\OOObsc}{
	\begin{tabexample}\label{egtab:equal_data_both_sc}\end{tabexample}
	&
	\begin{forest}
		compact tower
		[\emptyset, % \kappa_r
		[\emptyset, baseline
		[{\yng(1)}]
		[\emptyset]
		]
		[\emptyset,
		[\emptyset]
		[{\yng(1)}]
		] 
		]
		\node at (current bounding box.south)[below=0ex]{\partbypartssize $= (4,2,1,1)$};
	\end{forest}
	&
	\begin{forest}
		compact tower
		[\emptyset, % \kappa_r
		[\emptyset, baseline
		[\emptyset]
		[{\yng(1)}]
		]
		[\emptyset,
		[{\yng(1)}]
		[\emptyset]
		] 
		]
		\node at (current bounding box.south)[below=0ex]{\partbypartssize $= (3,3,2)$};
	\end{forest}
}

\newcommand{\OOOsc}{
	\begin{tabexample}\label{egtab:equal_data_onesc}\end{tabexample}
	&
	\begin{forest}
		compact tower
		[\emptyset, % \kappa_r
		[\emptyset, baseline
		[\emptyset]
		[\emptyset]
		]
		[\emptyset,
		[\emptyset]
		[{\yng(1)}]
		] 
		]
		\node at (current bounding box.south)[below=0ex]{\partbypartssize $= (4)$};
	\end{forest}
	&
	\begin{forest}
		compact tower
		[\emptyset, % \kappa_r
		[{\yng(1)}, baseline
		% [\emptyset]
		% [\emptyset]
		]
		[{\yng(1)} %,
		% [\emptyset]
		% [\emptyset]
		] 
		]
		\node at (current bounding box.south)[below=4ex]{\partbypartssize $= (2,2)$};
	\end{forest}
}

\newcommand{\IOIsc}{
	\begin{tabexample}\label{egtab:diff_2part_diff_height_onesc}\end{tabexample}
	&
	\begin{forest}
		compact tower
		[{\yng(3,2,1)} % \kappa_r
		[{\yng(3,2,1)}, baseline]
		[\emptyset, minimum size={19pt}]
		]
		\node at (current bounding box.south)[below=0ex]{\partbypartssize $= (9,6,3)$};
	\end{forest}
	&
	\begin{forest}
		compact tower
		[{\yng(3,2,1)} % \kappa_r
		[{\yng(2,1)}, baseline]
		[{\yng(2,1)}] 
		]
		\node at (current bounding box.south)[below=0ex]{\partbypartssize $= (7,4,2,2,1,1,1)$};
	\end{forest}
}

\newcommand{\fourlayertower}{%
	\!\begin{forest}
		really compact tower
		[{\yng(1)},
		[\emptyset, baseline
		[\emptyset,
		[\emptyset]
		[\emptyset]
		]
		[{\yng(1)},
		[{\yng(1)}]
		[\emptyset]
		]
		]
		[\emptyset
		[{\yng(1)},
		[\emptyset]
		[\emptyset]
		]
		[\emptyset,
		[\emptyset]
		[\emptyset]
		]
		]
		]
		\node at (current bounding box.south)[below=0ex]{\partbypartssize $= (13,1,1,1,1)$};
	\end{forest}\!}

\newcommand{\IOInsc}{
	\begin{tabexample}\label{egtab:diff_2part_diff_height_neither_sc}\end{tabexample}
	&
	\fourlayertower
	&
	\flexbox[c]{\begin{forest} % this flexbox ensures the \mu column has the same width as the \lambda column
			compact tower
			[{\yng(1)},
			[\emptyset, baseline
			[\emptyset]
			[{\yng(2,1)}]
			]
			[\emptyset,
			[\emptyset]
			[{\yng(1)}]
			] 
			]
			\node at (current bounding box.south)[below=2ex]{\partbypartssize $= (9,3,3,2)$};
	\end{forest}}
	{\fourlayertower}
}

\newcommand{\OIIsc}{
	\begin{tabexample}\label{egtab:diff_block_diff_height_onesec}\end{tabexample}
	&
	\begin{forest}
		compact tower, for tree={minimum size={15pt}}
		[{\yng(1)},
		[{\yng(1)}, baseline]
		[\emptyset]
		]
		\node at (current bounding box.south)[below=0ex]{\partbypartssize $= (3)$};
	\end{forest}
	&
	\begin{forest}
		compact tower
		[{\yng(2,1)},
		[\phantom{\emptyset}, baseline, no edge]
		% [\emptyset] 
		]
		\node at (current bounding box.south)[below=0ex]{\partbypartssize $= (2,1)$};
	\end{forest}
}

\newcommand{\OIInsc}{
	\begin{tabexample}\label{egtab:diff_block_diff_height_neithersc}\end{tabexample}
	&
	\begin{forest}
		compact tower, for tree={minimum size={15pt}}
		[{\yng(1)},
		[{\yng(2,1)}, baseline]
		[\emptyset, minimum size={14pt}]
		]
		\node at (current bounding box.south)[below=4ex]{\partbypartssize $= (5,2)$};
	\end{forest}
	&
	\begin{forest}
		compact tower
		[{\yng(2,1)},
		[\emptyset, baseline
		[\emptyset]
		[\emptyset]
		]
		[\emptyset,
		[\emptyset]
		[{\yng(1)}]
		] 
		]
		\node at (current bounding box.south)[below=0ex]{\partbypartssize $= (6,1)$};
	\end{forest}
}

\newcommand{\IIOnsc}{
	\begin{tabexample}\label{egtab:diff_2part_diff_block_neither_sc}\end{tabexample}
	&
	\begin{forest}
		compact tower, for tree={minimum size={15pt}}
		[{\yng(1)},
		[{\yng(3,2,1)}, baseline]
		[{\yng(1)}, minimum size={19pt}]
		]
		\node at (current bounding box.south)[below=0ex]{\partbypartssize $= (7,4,2,2)$};
	\end{forest}
	&
	\begin{forest}
		compact tower
		[{\yng(2,1)},
		[{\yng(3,2,1)}, baseline]
		[\emptyset, minimum size={19pt}] 
		]
		\node at (current bounding box.south)[below=0ex]{\partbypartssize $= (3,3,2,2,2,1,1,1)$};
	\end{forest}
}

\newcommand{\tabexbox}[1]{\parbox{45pt}{\small #1}}
\newcolumntype{E}{>{\collectcell\tabexbox}r<{\endcollectcell}}
\newcolumntype{t}{>{\tiny}c}

% \vspace{-0.04cm}

\begin{table}[H]
\captionsetup{width=15cm}
	\caption{
		Differences in pairs of irreducible characters of \(A_n\) with equal vanishing on \(Q\).
	}
	\centering
    \vspace{-5pt}
	\begin{tabular}[t]{CCCt@{\hskip 10pt}E@{\hskip -5pt}C@{\hskip -5pt}C} \toprule
		\multicolumn{3}{c}{\small Difference in...} &
		\multirow{3}{*}{\footnotesize \makecell{Self-conjugacy \\ of labelling \\ partitions}} &
		&\multicolumn{2}{c}{\footnotesize Example}
		\\[1pt]
		\multicolumn{1}{c}{\tiny \makecell{\(2\)-adic \\ valuation \\ of degree}}\!\! &
		\multicolumn{1}{c}{\tiny \makecell{defect \\ of \(2\)-block}}\!\! &
		\multicolumn{1}{c}{\tiny \makecell{\(2\)-height}}\! &&& \lambda & \mu
		\\
		\midrule %----------------------------
		\multirow{3}{*}[-65pt]{0} & \multirow{3}{*}[-65pt]{0} & \multirow{3}{*}[-65pt]{0}
		& \makecell{neither \\ self-conjugate} & \OOOnsc \\
		\nudge
		&&& \makecell{both\\ self-conjugate} & \OOObsc \\
		\nudge
		&&& \makecell{exactly one \\ self-conjugate} & \OOOsc \\
		\halfnudge
		\midrule %----------------------------
		\halfnudge
		\multirow{2}{*}[-39pt]{1} & \multirow{2}{*}[-39pt]{0} & \multirow{2}{*}[-39pt]{1}
		& \makecell{neither \\ self-conjugate} & \IOInsc \\
		\nudge
		&& 
		& \makecell{exactly one \\ self-conjugate} & \IOIsc \\
		\halfnudge
		\midrule %----------------------------
		\halfnudge
		\multirow{2}{*}[-29pt]{0} & \multirow{2}{*}[-29pt]{1} & \multirow{2}{*}[-29pt]{1}
		& \makecell{neither \\ self-conjugate} & \OIInsc \\
		\nudge
		&&
		& \makecell{exactly one \\ self-conjugate} & \OIIsc \\
		\halfnudge
		\midrule %----------------------------
		\halfnudge
		1 & 1 & 0 & \makecell{neither \\ self-conjugate} & \IIOnsc \\ \halfnudge
		\bottomrule
	\end{tabular}
	\label{tab:Q-vanishing_pairs}
\end{table}

\bigskip
\section*{Acknowledgements}
We thank Gabriel Navarro for stimulating conversations on the subject of this article.
We are grateful to an anonymous referee for their comments on an earlier version of this paper, and in particular for pointing out the containment-inequality interpretation of our results now included in the statements of \Cref{thm:main-sym-compare-pparts} and \Cref{thm:main-sym-compare-allp}.

The first author acknowledges support from the European Union -- Next Generation EU, M4C1, CUP B53D23009410006, PRIN 2022; and from the INDAM-GNSAGA Project CUP E53C24001950001.
% The first author's research is funded by the European Union -- Next Generation EU, M4C1, CUP B53D23009410006, PRIN 2022 - 2022PSTWLB \textit{Group Theory and Applications}.
% As a member of the GNSAGA, he is grateful for the support of the {\em Istituto Nazionale di Alta Matematica}.
The third author was affiliated with Okinawa Institute of Science and Technology (OIST) Graduate University while much of this work was carried out.
The first and second authors thank the Heilbronn Institute for Mathematical Research for a Small Grant which supported the first author's visit to the University of Birmingham during which part of this work was done.
The first and third authors also thank the Algebra research group at the School of Mathematics, University of Birmingham for their hospitality during their visits.
% \vspace{-0.04cm}
% \vspace{-1cm}

\bigskip

%----------------------------------------------------------------------
% \section*{Acknowledgements}
% We thank Gabriel Navarro for stimulating conversations on the subject of this article.
% We are grateful to an anonymous referee for their comments an earlier version of this paper, and in particular for pointing out the containment-inequality interpretation of our results now included in the statements of \Cref{thm:main-sym-compare-pparts} and \Cref{thm:main-sym-compare-allp}.

% The first author acknowledges support from the European Union -- Next Generation EU, M4C1, CUP B53D23009410006, PRIN 2022; and from the INDAM-GNSAGA Project CUP E53C24001950001.
% % The first author's research is funded by the European Union -- Next Generation EU, M4C1, CUP B53D23009410006, PRIN 2022 - 2022PSTWLB \textit{Group Theory and Applications}.
% % As a member of the GNSAGA, he is grateful for the support of the {\em Istituto Nazionale di Alta Matematica}.
% The third author was affiliated with Okinawa Institute of Science and Technology (OIST) Graduate University while much of this work was carried out.
% The first and second authors thank the Heilbronn Institute for Mathematical Research for a Small Grant which supported the first author's visit to the University of Birmingham during which part of this work was done.
% The first and third authors also thank the Algebra research group at the School of Mathematics, University of Birmingham for their hospitality during their visits.
% % \vspace{-0.04cm}

%----------------------------------------------------------------------

\end{document}